\documentclass[a4paper,12pt]{amsart}
\usepackage{amsfonts,amsmath,amssymb}
\usepackage{amssymb}
\usepackage{verbatim}
\usepackage[dvips]{graphicx}
\usepackage{psfrag}
\usepackage{epsfig}
\usepackage{color}
\usepackage{graphics}
\usepackage[lofdepth,lotdepth]{subfig}
\theoremstyle{plain}
\newtheorem{theorem}{Theorem}[section]
\newtheorem{lemma}[theorem]{Lemma}
\newtheorem{proposition}[theorem]{Proposition}
\newtheorem{corollary}[theorem]{Corollary}
\newtheorem{remark}[theorem]{Remark}
\newtheorem{definition}[theorem]{Definition}

\theoremstyle{definition}
\theoremstyle{remark}
\numberwithin{equation}{section}
\mathchardef\emptyset="001F

\newcommand{\e}{\varepsilon}
\newcommand{\Om}{\Omega}
\newcommand{\om}{\omega}
\newcommand{\weakst}{\stackrel{\ast}{\rightharpoonup}}

\newcommand{\R}{{\mathbb R}}
\newcommand{\F}{{\mathcal F}}

\renewcommand{\L}{{\mathcal L}}
\newcommand{\I}{{\mathcal I}}
\newcommand{\E}{{\mathcal E}}
\newcommand{\W}{{\mathcal W}}

\newcommand{\A}{{\mathcal A}}

\newcommand{\Z}{{\mathbb Z}}
\newcommand{\N}{{\mathbb N}}

\newcommand{\M}{{\mathbb M}}
\renewcommand{\H}{{\mathcal H}}

\newcommand{\Mdd}{\M^{2\times 2}}
\newcommand{\Mnn}{\M^{N\times N}}

\newcommand{\dsp}{\displaystyle}
\newcommand{\rank}{\mathrm{rank}}
\newcommand{\dist}{\mathrm{dist}}

\newcommand{\co}{{\rm co}}

\newcommand{\ut}{\tilde u}
\newcommand{\w}{\mathrm{w}}
\newcommand{\wu}{\mathrm{v}}
\newcommand{\subdo}{ P}
\newcommand{\Subdo}{\Pi}
\newcommand{\ga}{\gamma}
\newcommand{\Ga}{\Gamma}
\newcommand{\LL}{\L}
\newcommand{\TT}{\mathcal T}
\newcommand{\NN}[2]{#1,\,#2\ {\rm NN}} 
\newcommand{\NNN}[2]{#1,\,#2\ {\rm NNN}}

\newcommand{\wto}{\rightharpoonup}

\newcommand{\be}{\begin{equation}}
\newcommand{\ee}{\end{equation}}
\newcommand{\bes}{\begin{equation*}}
\newcommand{\ees}{\end{equation*}}
\newcommand{\bea}{\begin{eqnarray}}
\newcommand{\eea}{\end{eqnarray}}
\newcommand{\beas}{\begin{eqnarray*}}
\newcommand{\eeas}{\end{eqnarray*}}
\newcommand{\integ}[3]{\int_{#1} #2 \, \d #3}
\newcommand{\integlin}[4]{\int_{#1}^{#2} #3 \, \d #4}

\renewcommand{\d}{\mathrm{d}}
\newcommand{\D}{\nabla} 

\newcommand{\eps}{\varepsilon}

\begin{document}
\title[A discrete to continuum analysis of dislocations in nanowires]{A discrete to continuum analysis of dislocations in nanowire heterostructures}
\author{Giuliano Lazzaroni}
\author{Mariapia Palombaro}
\author{Anja Schl\"omerkemper}
\address{G.\ Lazzaroni, A.\ Schl\"omerkemper: University of W\"urzburg, Department of Mathematics, Emil-Fischer-Stra\ss{}e 40, 97074 W\"urzburg, Germany}
\email[G.\ Lazzaroni]{giuliano.lazzaroni@uni-wuerzburg.de}
\email[A.\ Schl\"omerkemper]{anja.schloemerkemper@mathematik.uni-wuerzburg.de}
\address{M.\ Palombaro: University of L'Aquila, Department of Information Engineering, Computer Science and Mathematics,
Via Vetoio 1 (Coppito), 67100 L'Aquila, Italy}
\email[M.\ Palombaro]{mariapia.palombaro@univaq.it}
\thanks{\today}
\begin{abstract}
\small{
Epitaxially grown heterogeneous nanowires present dislocations at the interface between the phases if their radius is big. 
We consider  a corresponding variational discrete model with quadratic pairwise atomic interaction energy. 
By employing the notion of Gamma-convergence and a geometric rigidity estimate, we perform 
a discrete to continuum limit and a dimension reduction to a one-dimensional system. Moreover, 
we compare a defect-free model and models with dislocations at the interface and show that the latter
are energetically convenient if the thickness of the wire is sufficiently large.
}
\end{abstract}
\maketitle
{\small
\keywords{\noindent {\bf Keywords:}
Nonlinear elasticity, Discrete to continuum, Dimension reduction, Rod theory, Geometric rigidity, Non-inter\-pen\-e\-tra\-tion, Gamma-con\-vergence, 
Crystals,  Dislocations, Heterostructures.
}
\par
\subjclass{\noindent {\bf 2010 MSC:}
74B20, 
74K10, 
74N05, 
70G75, 
49J45. 
}
}
\bigskip
\section*{Introduction}
Nanowire heterostructures are of interest in technical applications such as semiconductor electronics and optoelectronics. 
The epitaxial growth of such nanowires is a process of deposition of atoms of a certain crystal on top of a substrate made from a different crystal.
The devices consist then of two phases which meet at an interface that lies on the small cross-section of the wire.
If the radius of the wire is rather small, the wire can relieve a large amount of the strain energy induced by the presence of different phases by elastic deformation 
and then does not display dislocations; this is advantageous since dislocations interfere with electronic properties. 
Heterostructured nanowires are found to be promising devices since 
they can be grown defect-free 
more readily than films, see e.g.\ the reviews \cite{Kavanagh,Schmidtetal2010}. 
\par
The aim of this work is to provide further insight into the understanding of the (non)-occurrence of dislocations, 
see \cite{LPSpro} for an abridged version.
Ertekin et al.\ \cite{egcs} were the first who recognised that nanowire heterostructures should be defect-free 
if the radius is small enough. They proposed a variational principle in the context of linearised elasticity, 
which was later rigorously justified by M\"uller and Palombaro \cite{mp} in terms of $\Ga$-convergence. 
In \cite{mp} the authors consider a fully continuum model in the framework of nonlinear elasticity with dislocations
and study the $\Ga$-limit as the dimension of the system reduces to one,
extending the result of \cite{mm} for elastic multiphase materials.
In this work we show that their results can be recovered starting from a discrete model,
thus giving a microscopic justification of the continuous description. In our setting 
the formula for the energy contribution due to the lattice mismatch across the interface
solely depends on microscopic parameters and thus can in principle be computed from the data at the lattice scale.
\par
We assume that the nanowire heterostructure consists of two materials which are crystals with the same lattice structure,
but with different lattice distances. The interface between the two materials is assumed to be flat.
The resulting lattice mismatch can in principle be compensated
either by elastic strain or by introduction of dislocations at the interface:
the aim of this paper is to analyse which of the two cases is energetically preferred by the system.
In the former case, the crystal will be defect-free and the atomic bonds in the vicinity of the interface 
will be strained in order to pass from one lattice distance to the other;
in the latter situation, the number of lines of atoms
will be different in the two phases, see Section~\ref{sec:not} for details.
\par
Most of this paper is devoted to the two-dimensional case for the sake of simplicity and to prepare the three-dimensional case studied in Section~\ref{3d}. In order to simplify the presentation and to focus on the most important aspects of the problem,
we assume that the total energy depends only on nearest-neighbour interactions
and that these interactions are harmonic. Generalisations of these and other assumptions are discussed in Section~\ref{sec:rmk}. 
We suppose also that the nearest-neighbour bonds divide the space into rigid cells,
so we consider for instance the hexagonal (equilateral triangular) 
Bravais lattice in two dimensions and the face-centred cubic lattice in three dimensions,
related to e.g.\ silicon nanowires \cite{Schmidtetal2010}.
Finally, we postulate the non-interpenetration condition,
requiring that the discrete deformations preserve the orientation of each cell.
The same assumptions were made e.g.\ in~\cite{BSV},
see also \cite{FT} for the case of a square lattice with second neighbours. Here they allow us to apply the rigidity estimate of 
Friesecke, James, and M\"uller \cite{fjm}.
\par
We prove that defect-free configurations are more expensive than dislocated configurations
if the thickness of the wire is sufficiently large.
More precisely, in Section~\ref{sec:minimum} 
we compare the energy for a set of configurations that are allowed to contain defects at the interface only 
and that are defect-free, respectively. We assume that the crystal chooses that configuration
which minimises the total interaction energy between the atoms. 
For this, the nearest neighbours in the lattice should be defined in the deformed configuration.
However, following the usual approach in the mechanics of discrete systems
we say that the nearest neighbours in the deformed configuration are determined by the nearest neighbours in the reference configuration.
Therefore, in the study of defect-free and dislocated configurations we consider models which differ already in the reference configuration. 
The nearest neighbours are then chosen based on a Delaunay triangulation, as done e.g.\ in \cite{ACG} for stochastic lattices.
\par
For each of such models, we study the minimisers of the total interaction energy.
In Section~\ref{sec:lim} we perform a discrete to continuum analysis as the lattice distance tends to zero,
employing the notion of $\Ga$-convergence;
for its definition, its most relevant properties, and applications to discrete problems we refer e.g.\ to \cite{Braides,BG}.
The model is constructed in such a way that the lattices converge to a line as the lattice distance tends to zero,
so the dimension reduces to one.
Atomistic systems with dimension reduction were previously treated e.g.\ in \cite{ABC,FJ,Schm08} 
in the case of defect-free lattices converging to thin films.
However in those works the scaling of the energy is different from ours, see Remark \ref{rmk:abc};
more precisely, the $\Ga$-limit of our functionals (in the defect-free case)
is related to the first order $\Ga$-limit of the functionals studied in \cite{ABC,Schm08}.
In the case of screw dislocations, a different approach based on discrete to continuum techniques
was proposed in \cite{superPonsi}.
\par
The discrete to continuum limit provides a characterisation of the minimum cost necessary to compensate
the lattice mismatch in defect-free or dislocated models,
in dependence on a mesoscale parameter $k$. The thickness of the nanowire
is a multiple of the lattice distance, with $k$ denoting the proportionality coefficient.
We study such cost as $k$ varies, showing that it scales as $k^{N}$ for the minimal defect-free configuration
and as $k^{N-1}$ if dislocations are suitably introduced:
this is presented in Section \ref{sec:gamma} when the spatial dimension $N$ is two and in Section \ref{3d} for $N=3$.
As a consequence, dislocations are favoured if the thickness is sufficiently large (Corollary \ref{coro}).
It would be also of interest to determine the threshold $k_c$ such that 
the defect-free model is energetically preferred for every $k\le k_c$;
this however is to be left for future investigation.
\par
\section{Setting of the model}\label{sec:not}
We define an atomistic model for two-dimensional nanowires; in Section~\ref{3d} we will discuss a generalisation to three dimensions. 
We consider longitudinally heterostructured nanowires, with two crystalline phases corresponding to two types of atoms with different equilibrium distance.
In our simplified model, each of the two phases has a given, possibly different lattice distance also in the reference configuration, 
which leads to the presence of dislocations at the interface (if the nanowire's thickness is sufficiently large), see Figure~\ref{fig-disl}.
\par
The energy associated to a deformation of the atomistic system depends on four quantities, $\eps$, $\lambda$, $k$, and $\rho$.
The variable $\eps>0$ denotes the equilibrium distance between atoms of the first material.
We study a discrete to continuum limit, that is, the passage to the limit as $\eps$ tends to zero.
\par
The quantity $\lambda\in(0,1)$ is the ratio of the equilibrium distances of the two materials
(in the deformed configuration). It is a datum of the problem and will be constant throughout the paper.
Since the equilibrium distance of the first material is $\eps$, the atoms of the second one have equilibrium distance $\lambda\eps$.
\par
The body depends on another parameter $k\in\N$, $k\ge1$,
related to the thickness of the nanowire.
More precisely, in two dimensions the reference configuration is a parallelogram 
whose sides have length $2L$ and $k\eps$, respectively, with $L>0$ constant.
While $\eps\to0^+$, the thickness $k$ is fixed, so one has dimension reduction to a line.
After this passage to the limit, we will consider in particular the asymptotics for large $k$.
\par
Finally, we use a variable $\rho$ to allow for different geometries of the nearest neighbours
and in particular for dislocations.
In our model the atomic spacing in the reference configuration depends on the phase, too.
The parameter $\rho\in(0,1]$ stands for the ratio of the lattice distances of the two phases in the reference configuration
(the most interesting case being $\rho\in[\lambda,1]$).
\par
If $\rho=1$ we recover a defect-free model, where the coordination number (i.e., the number of nearest neighbours of an internal atom)
is constant in the lattice, see Figure \ref{fig-def-free}. For $\rho\neq1$ (and $k$ sufficiently large) the coordination number is not constant,
see Figure \ref{fig-disl}; when this happens, the crystal contains dislocations, see Remark \ref{rmk:Del1} and \ref{rmk:Del2} for more details.
\par
We now make precise the definition of the lattice and of the nearest neighbours, first of all in the unbounded case.
In each of the phases we consider a two-dimensional hexagonal (also called equilateral triangular) Bravais lattice, 
which is rigid already for nearest-neighbour interactions (an essential property for our result).
For simplicity, we assume the interface to be a straight line.
In Section \ref{more-general} and \ref{3d} we outline how the following results can be extended to other lattices in two or three dimensions.
\par
We set
$\w_1:=(1,0)$, $\w_2:=(\frac{1}{2},\frac{\sqrt{3}}{2})$, $\w_3:=\w_2-\w_1$, and
\bea  \label{1307251}
\LL_1^- &:=& \{ \xi_1\w_1+\xi_2\w_2\colon \xi_1,\xi_2\in\Z\,,\ \xi_1<0\} \,,\\  \label{1307252}
\LL_\rho^+ &:=& \{ \xi_1\w_1+\xi_2\w_2\colon \xi_1,\xi_2\in\rho\Z\,,\ \xi_1\ge0 \} \,,\\ \label{1307253}
\LL_\rho&:=&\LL_1^-\cup\LL_\rho^+\,.
\eea
For $\rho=1$, $\LL_1$ is the two-dimensional hexagonal Bravais lattice generated by the vectors $\w_1$ and $\w_2$.
The interfacial atoms lie then on the two lines 
$\{-\w_1+\xi_2\w_2\colon\xi_2\in\R\}$ and $\{\xi_2\w_2\colon\xi_2\in\R\}$.
It is possible to consider a different distance between the two interfacial lines, see Section \ref{more-general}. 
\par
\subsection{Triangulation of lattices}\label{subsec:tri}
We now discuss the choice of the nearest neighbours in $\LL_\rho$; 
to this aim, we employ the classic notions of Voronoi cell and Delaunay triangulation (Figure~\ref{fig-del}), see e.g.~\cite{Okabe}.
For the reader's convenience we state here the definitions of these diagrams, adapting them to our context.
When $\rho=1$ the following construction reduces to the standard notion of nearest neighbours in the two-dimensional hexagonal Bravais lattice,
i.e., the nearest neighbours of an atom are the six atoms with minimal distance.
\par
\begin{figure}[b]
\centering
\subfloat[]{
\includegraphics[width=.3\textwidth]{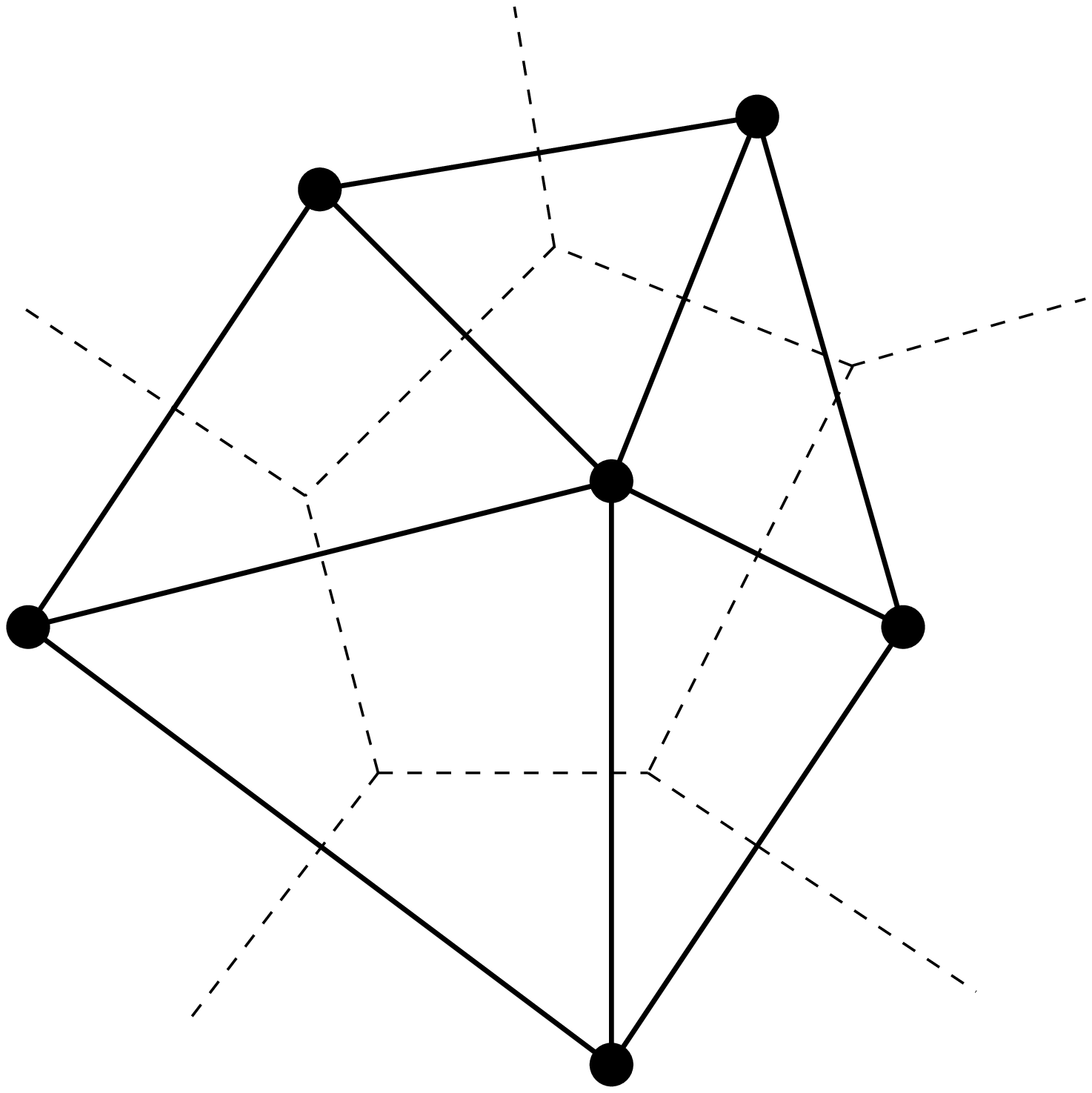}
\label{fig:del1}
}
\hspace{.1\textwidth}
\subfloat[]{
\includegraphics[width=.3\textwidth]{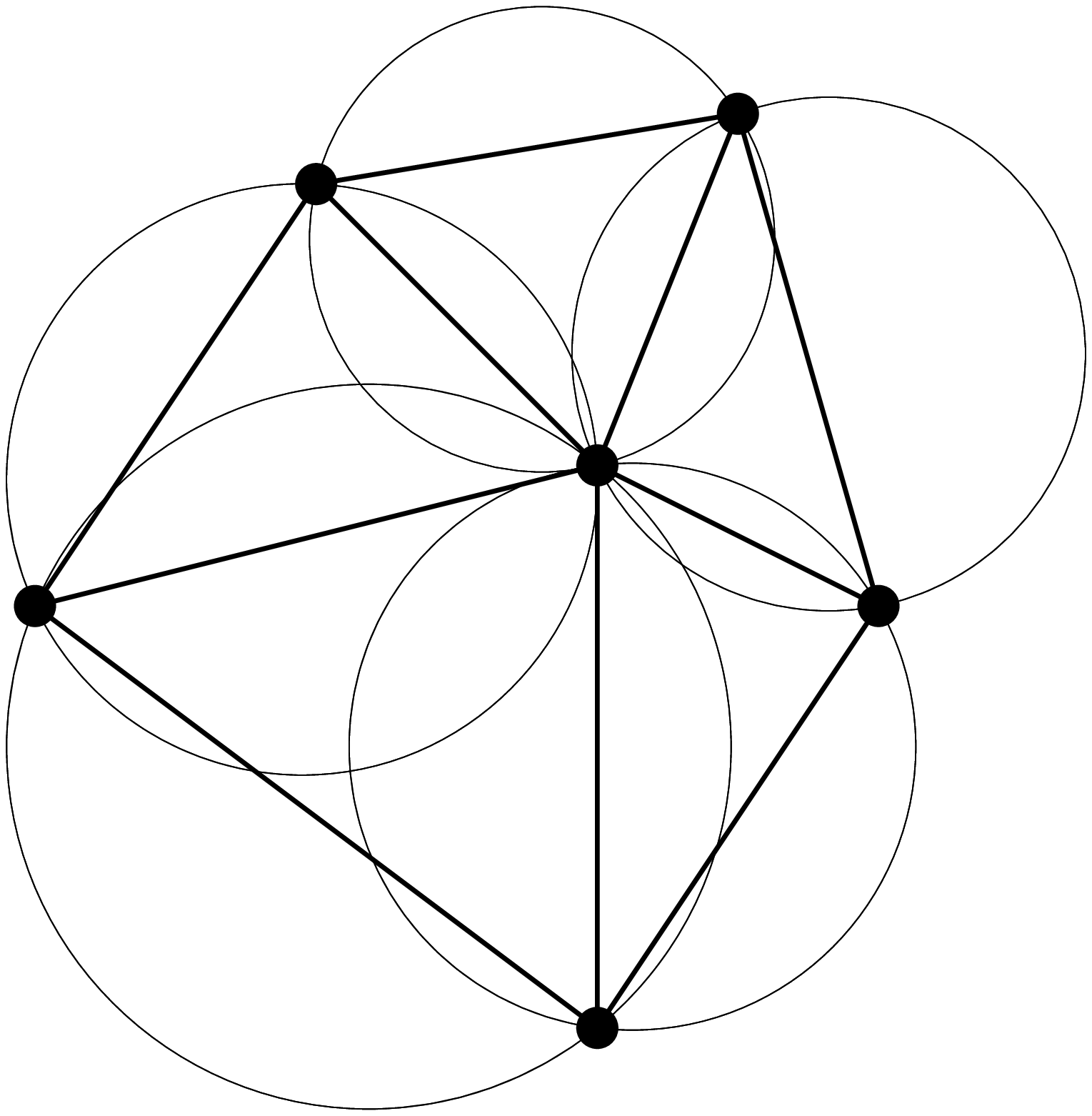}
\label{fig:del2}
}
\caption{(\textsc{a}) Voronoi diagram (dashed) and Delaunay triangulation (bold) for a set of points.
(\textsc{b}) Circumcircles to the elements of the triangulation.}\label{fig-del}
\end{figure}
For later reference, the definitions are given in a general case of a lattice in $\R^N$, $N\ge2$.
Let $\LL\subset\R^N$ be a countable set of points such that there exist $R,r>0$ with
$\inf_{x\in\R^N} \# \LL\cap B(x,R)\ge1$ and
$|x-y|\ge r$ for every $x,y\in\LL$, $x\neq y$,
where $B(x,R):=\{y\in\R^N\colon |x-y|<R\}$.
\begin{definition}[Voronoi cells] \label{def:Voronoi}
The Voronoi cell of a point $x\in\LL$ is the set
$$
C(x):= \{ z\in\R^N\colon |z-x|\le|z-y| \ \forall\, y\in\LL \} \,.
$$
The Voronoi diagram associated with $\LL$ is the partition $\{C(x)\}_{x\in\LL}$.
\end{definition}
The Voronoi diagram associated with a lattice is unique 
and determines a unique Delaunay pretriangulation.
\begin{definition}[Delaunay pretriangulation]\label{def:Del-pre}
The Delaunay pretriangulation associated with $\LL$ is a partition of $\R^N$ in open nonempty hyperpolyhedra with vertices in $\LL$,
such that two points $x,y\in\LL$ are vertices of the same hyperpolyhedra if and only if $C(x)\cap C(y)\neq\emptyset$.
\end{definition}
Notice that in the previous definition, two points $x,y$ are extrema of one of the edges of the same hyperpolyhedron 
if and only if $\H^{N-1}(C(x)\cap C(y))>0$, where $\H^{N-1}$ denotes the $(N{-}1)$-dimensional measure
(i.e., $C(x)$ and $C(y)$ have a common side if $N=2$ or a common facet if $N=3$).
The Delaunay pretriangulation is called Delaunay triangulation if all of its cells are $N$-simplices;
otherwise, the Delaunay pretriangulation can be refined in different ways in order to obtain a triangulation.
In both cases those triangulations satisfy the next property, which follows from~\cite[Property~V7]{Okabe}.
\begin{definition}[Delaunay property] \label{def:Del}
Let $\TT$ be a triangulation associated with $\LL$, i.e., a partition of $\R^N$ in open nonempty $N$-simplices with vertices in $\LL$.
We say that $\TT$ has the Delaunay property if, for every simplex of $\TT$, its circum-hypersphere contains no points of $\LL$ in its interior.
\end{definition}
In general, a lattice $\LL$ has a unique triangulation with the Delaunay property if and only if $\LL$ is nondegenerate according to the next definition
\cite[Property~D1]{Okabe}.
\begin{definition}[Nondegenerate lattice]
We say that $\LL\subset\R^N$ is nondegenerate if the following property holds:
if $x_1,\dots,x_{N+1}\in\LL$ are such that no points of $\LL$ lie in the interior of the circum-hypersphere of the simplex with vertices $x_1,\dots,x_{N+1}$,
then no further points of $\LL$ lie on that circum-hypersphere.
Otherwise the lattice is called degenerate.
\end{definition}
We now come back to the two-dimensional context.
In order to fix a triangulation of $\LL_\rho$ fulfilling the Delaunay property also in the degenerate case,
we determine concretely the Delaunay structures associated to a lattice of the type $\LL_\rho$.
Because of the definition of hexagonal Bravais lattice, it is clear that the possible degeneracies of $\LL_\rho$ are situated across the interface, i.e.,
the points in degenerate position are always among those of the type
$$
P_i:=-\w_1+i\w_2\,,\quad Q_j:=j\rho\w_2\quad \text{for}\ i,j\in\Z \,.
$$
The lattice $\LL_\rho$ is degenerate if and only if there exist $i,j\in\Z$ such that the points
$P_i$, $P_{i+1}$, $Q_j$, $Q_{j+1}$ lie on the same circle;
this happens if and only if $\rho=2i/(2j+1)$.
\par
\begin{figure}
\centering
\psfrag{pj}{$P_i$}
\psfrag{pj+}{\hspace{-.4cm}$P_{i+1}$}
\psfrag{qj}{$Q_j$}
\psfrag{qj+}{$Q_{j+1}$}
\subfloat[]{
\includegraphics[width=.3\textwidth]{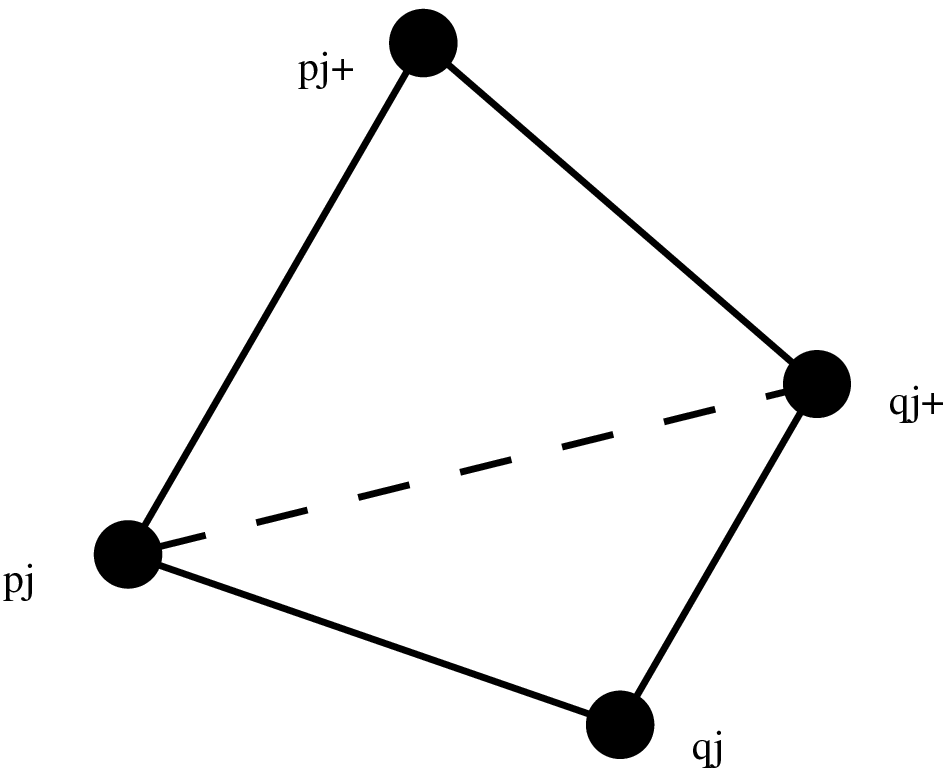}
\label{fig:degen1}
}
\hspace{.1\textwidth}
\subfloat[]{
\includegraphics[width=.3\textwidth]{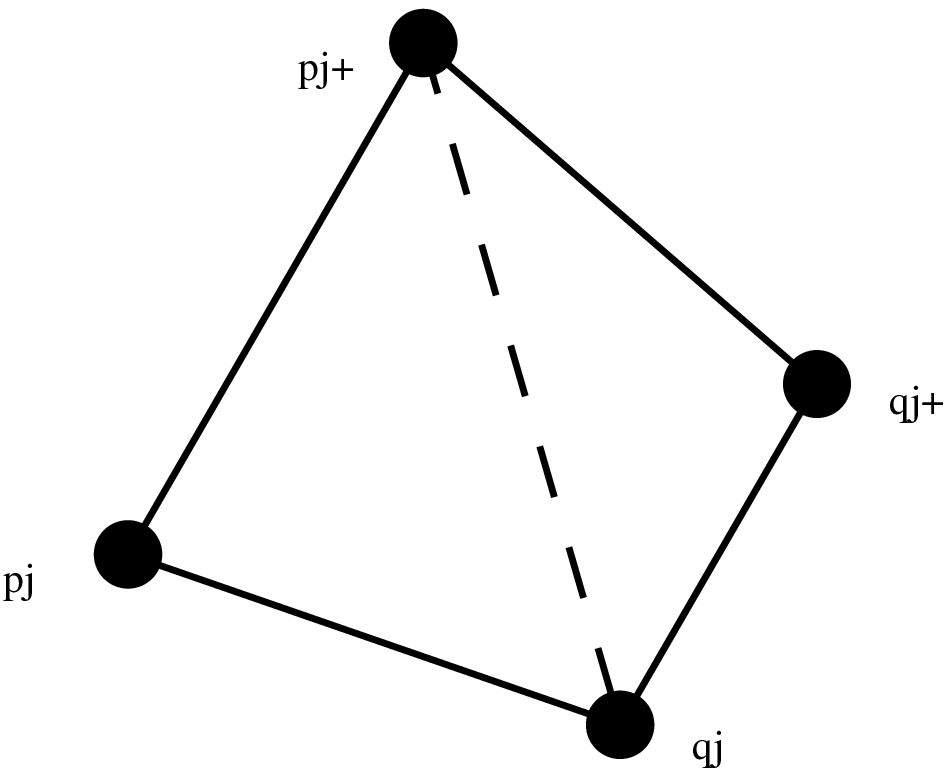}
\label{fig:degen2}
}
\caption{Delaunay triangulations associated with a degenerate configuration.}
\label{fig-degen}
\end{figure}
We are now in a position to fix a Delaunay triangulation associated with $\LL_\rho$.
If $\LL_\rho$ is nondegenerate, we fix $\TT_\rho$ to be the unique triangulation satisfying Definition~\ref{def:Del}.
In the opposite case, if $P_i$, $P_{i+1}$, $Q_j$, $Q_{j+1}$ are in degenerate position, there are two choices for the triangulation
of the isosceles trapezium $P_{i}P_{i+1}Q_jQ_{j+1}$, see Figure \ref{fig-degen}:
\begin{enumerate}
\item[(a)] $P_{i}Q_jQ_{j+1}$ and $P_iP_{i+1}Q_{j+1}$ are part of $\TT_\rho$, i.e., $P_i$ and $Q_{j+1}$ are neighbours and $P_{i+1}$ and $Q_{j}$ are not;
\item[(b)] $P_iP_{i+1}Q_j$ and $P_{i+1}Q_jQ_{j+1}$ are part of $\TT_\rho$, i.e., $P_{i+1}$ and $Q_{j}$ are neighbours and $P_{i}$ and $Q_{j+1}$ are not.
\end{enumerate}
Both the possibilities lead to a triangulation with the Delaunay property.
For every 4-tuple of points where this situation occurs, by convention we pick the triangulation (a); this fixes $\TT_\rho$.
Notice that $\dist(P_i,Q_{j+1})=\dist(P_{i+1},Q_j)$ since the lines $P_iP_{i+1}$ and $Q_jQ_{j+1}$ are parallel,
so the choice of the nearest neighbours cannot be based on a distance criterion.
\par
The definition of nearest neighbours follows.
\begin{definition}[Nearest neighbours]
Two points $x,y\in\LL_\rho$, $x\neq y$, are said to be nearest neighbours (and we write: $\NN{x}{y}$)
if they are vertices of one of the triangles of $\TT_\rho$.
\end{definition}
\begin{remark} \label{rmk:Del1}
From the construction of $\LL_\rho$ and $\TT_\rho$ it turns out that for $\rho<1$ two points $P_i$, $Q_j$ can interact only if $i-1<j\rho+\frac12<i+1$.
It follows that a point of type $P_i$ has at least two and at most $\lfloor2/\rho\rfloor+1$ neighbours of type $Q_j$;
a point of type $Q_j$ has at least one and at most two neighbours of type $P_i$
(where $\lfloor\cdot\rfloor$ denotes the integer part).
Moreover, a bond in the reference configuration is never longer than $\frac{\sqrt{7}}{2}$.
\par
Therefore, the total number of neighbours of a point of type $Q_j$ is either five or six.
In the former case, we say that the lattice has a dislocation at that point;
in the latter, we say that the point is regular, since six is the coordination number of the two-dimensional hexagonal Bravais lattice.
\end{remark}
\subsection{Reference configuration, admissible deformations, and interaction energy}
We now pass to bounded lattices. Given $L>0$, $\eps\in(0,1]$, and $k\in\N$, we define the parallelogram
\bes
\Om_{k\eps} := \{ \xi_1\w_1+\xi_2\w_2\colon \xi_1\in(-L,L)\,,\ \xi_2\in(0, k\eps) \} \,.
\ees
We introduce the discrete strip with lattice distance $\eps$,
\be \label{1307254}
\L_{\rho,\eps}(k):=(\eps\LL_\rho)\cap\overline\Om_{k\eps} \,,
\ee
see Figures \ref{fig-def-free} and \ref{fig-disl}.
In our model the material presents two phases:
therefore we define the subsets
\beas
\Om_{k\eps}^- &:=& \{ \xi_1\w_1+\xi_2\w_2\colon \xi_1\in(-L,0)\,,\ \xi_2\in(0, k\eps) \} \,,\\
\Om_{k\eps}^+ &:=& \{ \xi_1\w_1+\xi_2\w_2\colon \xi_1\in(0,L)\,,\ \xi_2\in(0, k\eps) \} \,,\\
\L_{1,\eps}^-(k) &:=& \{\xi_1\w_1+\xi_2\w_2\in\L_{\rho,\eps}(k)\colon \xi_1<0\} \,,\\
\L_{\rho,\eps}^+(k) &:=& \{\xi_1\w_1+\xi_2\w_2\in\L_{\rho,\eps}(k)\colon \xi_1\ge0\} \,.
\eeas
We define also
$$
\TT_{\rho,\eps}:=\{\eps T\colon T\in\TT_\rho\} \,.
$$
\par
\begin{figure}
\centering
\psfrag{L}{$2L$}
\psfrag{ke}{\hspace{-.5em}$k\eps$}
\psfrag{e}{\hspace{-.1em}$\eps$}
\psfrag{le}{\hspace{-.1em}$\eps$}
\includegraphics[width=.8\textwidth]{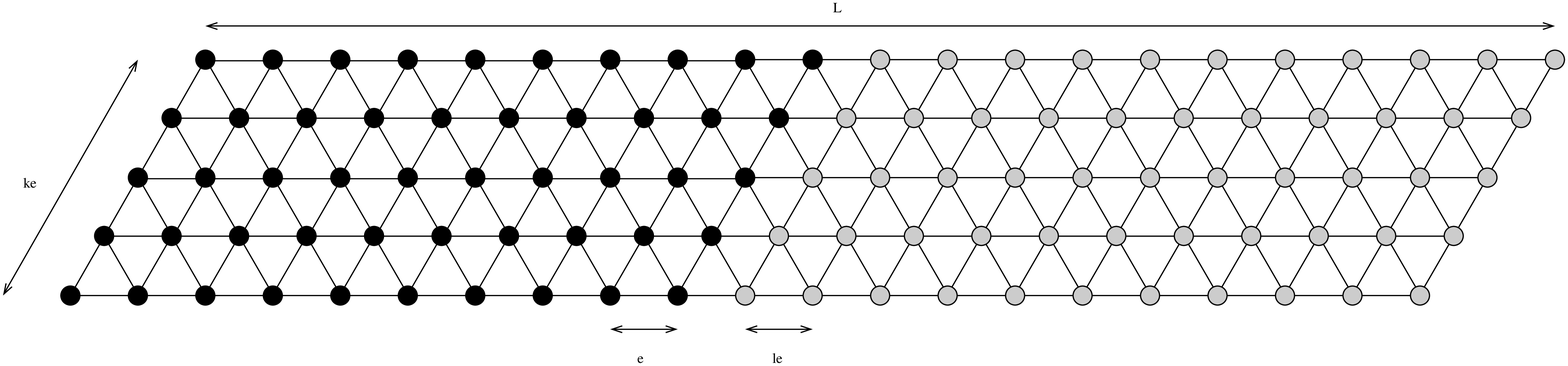}
\caption{A defect-free lattice of type $\LL_{1,\eps}(k)$.}
\label{fig-def-free}
\end{figure}
\begin{figure}
\centering
\psfrag{L}{$2L$}
\psfrag{ke}{\hspace{-.5em}$k\eps$}
\psfrag{e}{\hspace{-.1em}$\eps$}
\psfrag{le}{\hspace{-.4em}$\rho\eps$}
\includegraphics[width=.8\textwidth]{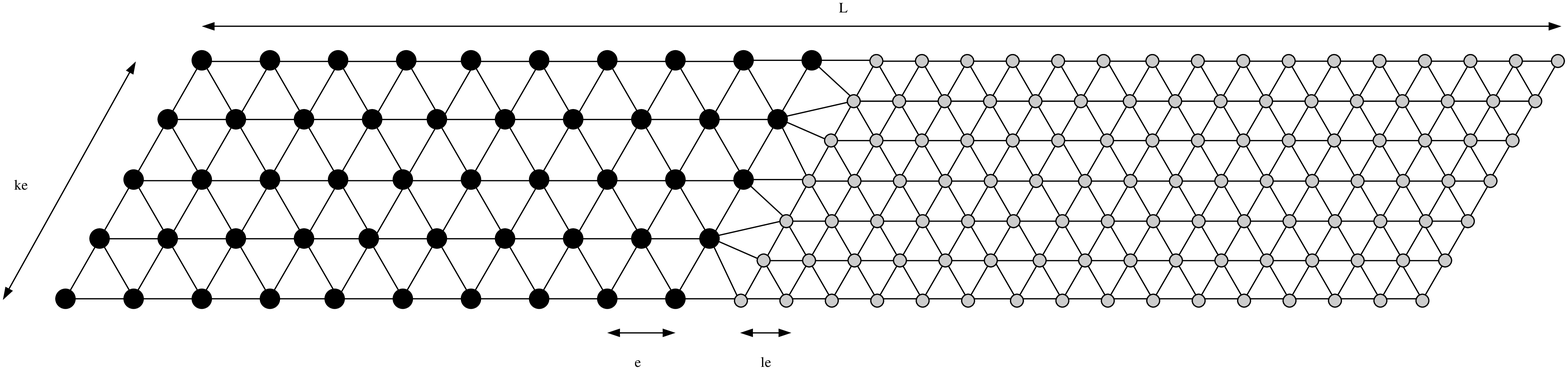}
\caption{A lattice with dislocations of type $\LL_{\rho,\eps}(k)$.}
\label{fig-disl}
\end{figure}
As for the nearest neighbours, we adopt the following notion: two points $x,y$ in any of the previous lattices
are nearest neighbours if $x/\eps$, $y/\eps$ fulfill the corresponding property in the lattice $\LL_\rho$.
\par
\begin{remark}\label{rmk:Del2}
The lattice $\L_{1,\eps}^-(k)$ contains $k+1$ lines of atoms parallel to $\w_1$, while
$\L_{\rho,\eps}^+(k)$ has $\lfloor k/\rho\rfloor+1$ lines.
The total number of dislocations of the lattice is then $\lfloor k/\rho\rfloor-k$,
which corresponds to the number of points $x\in\L_{\rho,\eps}^+(k)$
such that $x/\eps$ has five neighbours in $\LL_\rho$, see Remark \ref{rmk:Del1}.
\end{remark}
In the sequel of the paper we will often consider the rescaled domain $\frac{1}{\eps}\Om_{k\eps}$,
which converges, as $\eps\to0^+$, to the unbounded strip
\be\label{1308061}
\Om_{k,\infty}:= \{ \xi_1\w_1+\xi_2\w_2\colon \xi_1\in(-\infty,+\infty)\,,\ \xi_2\in(0,k) \} = \R\times\left(0,\tfrac{\sqrt3}{2}k\right) \,.
\ee
We define the associated lattice with lattice distances 1 and $\rho$,
\be
\L_{\rho,\infty}(k):=\LL_\rho\cap\overline\Om_{k,\infty} \,.
\ee
Also the infinite strip is divided into two subsets: 
\bea \label{1307231}
\L_{1,\infty}^-(k) &:=& \{\xi_1\w_1+\xi_2\w_2\in\L_{\rho,\infty}(k)\colon \xi_1<0\} \,,\\  \label{1307232}
\L_{\rho,\infty}^+(k) &:=& \{\xi_1\w_1+\xi_2\w_2\in\L_{\rho,\infty}(k)\colon \xi_1\ge0\} \,.
\eea
As before, two points in the previous lattices are said to be nearest neighbours if they are such in the lattice $\LL_\rho$.
\par
Every deformation $u_\eps$ of the lattice $\LL_{\rho,\eps}(k)$ is extended
by piecewise affine interpolation with respect to the triangulation $\TT_{\rho,\eps}$.
The set of admissible deformations is then
\be \label{adm}
\begin{split}
\A_{\rho,\eps}(\Om_{k\eps}):= \big\{ u_\eps\in C^0(\overline\Om_{k\eps};\R^2) \colon & u_\eps \ \text{piecewise affine,}\\ 
& \D u_\eps \ \text{constant on}\ \Om_{k\eps}\cap T\ \forall\, T\in\TT_{\rho,\eps}, \\
& \det \D u_\eps>0 \ \text{a.e.\ in}\ \Om_{k\eps} \big\} \,.
\end{split}
\ee
With a slight abuse of notation, the restriction of $u_\eps\in\A_{\rho,\eps}(\Om_{k\eps})$ to $\L_{\rho,\eps}(k)$ is still denoted by $u_\eps$.
As for the domain $\Om_{k,\infty}$, we define in an analogous way
\be\label{130806}
\begin{split}
\A_{\rho,\infty}(\Om_{k,\infty}):= \big\{ u\in C^0(\overline\Om_{k,\infty};\R^2) \colon & u \ \text{piecewise affine,}\\ 
& \D u \ \text{constant on}\ \Om_{k,\infty}\cap T\ \forall\, T\in\TT_{\rho}, \\
& \det \D u>0 \ \text{a.e.\ in}\ \Om_{k,\infty} \big\} \,.
\end{split}
\ee
In the last definitions we imposed a non-interpenetration condition: the Jacobian determinant is positive, 
so the deformations preserve the orientation. This is a usual requirement in the mechanics of atomistic systems  \cite{BSV,FT}, 
see also Figure~\ref{fig-nonint}.
\par
\begin{figure}[b]
\centering
\includegraphics[width=.6\textwidth]{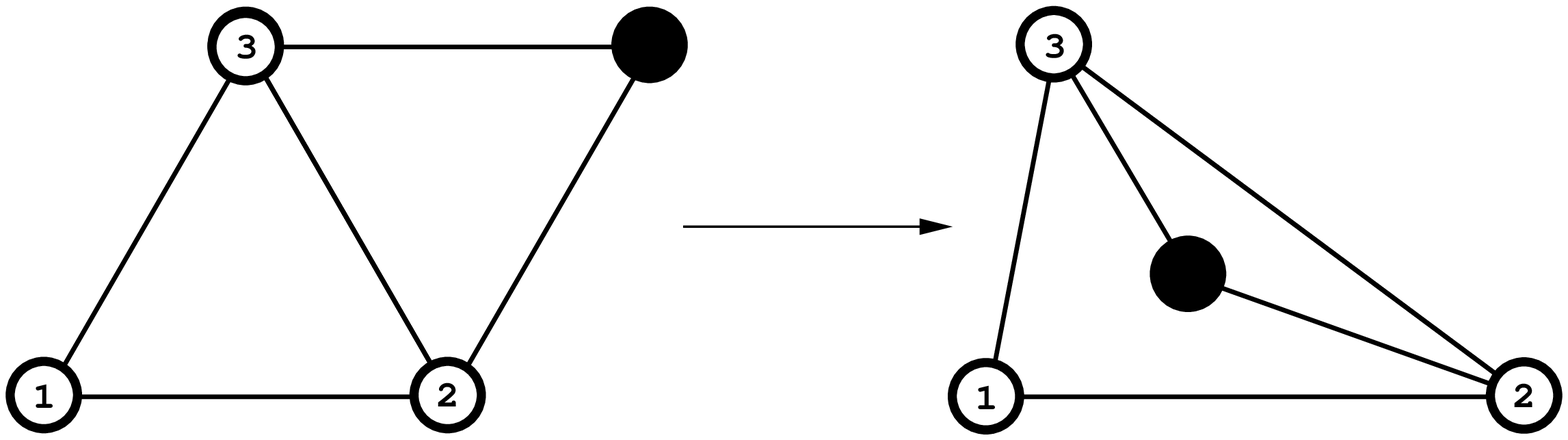}
\caption{This deformation of a system of four atoms is ruled out by the non-interpenetration condition,
since the induced piecewise affine extension has negative Jacobian determinant in one of the triangles.
No physical condition would prevent the black atom from being in the displayed position;
however, in such a case the nearest neighbours should be redefined, which is why this situation is excluded.}
\label{fig-nonint}
\end{figure}
\par
In the following we will consider also a scaling of the domain $\Om_{k\eps}$ to the domain $\Om_k$, which is independent of $\eps$.
Hence, given $u_\eps\in\A_{\rho,\eps}(\Om_{k\eps})$ we set
$\ut_\e (x):= u_\e(A_\e x)$, where 
$$
\dsp A_\e:= 
\begin{pmatrix}
1 & \frac{\e-1}{\sqrt{3}}
\\
0 & \e
\end{pmatrix}
\quad \text{and thus} \quad
\dsp A_\e^{-1}:= 
\begin{pmatrix}
1 & \frac{1-\e}{\eps\sqrt{3}}
\\[.5em]
0 & \frac1\e
\end{pmatrix} \,.
$$
The corresponding set of admissible deformations is
\be \label{tildeA}
\begin{split}
\widetilde\A_{\rho,\eps}(\Om_{k}):= \big\{ \ut_\eps\in C^0(\overline\Om_{k};\R^2) \colon & \ut_\eps \ \text{piecewise affine,}\\ 
& \D \ut_\eps \ \text{constant on}\ \Om_{k}\cap (A_\eps^{-1}T)\ \forall\, T\in\TT_{\rho,\eps}\,, \\
& \det \D\ut_\eps>0 \ \text{a.e.\ in}\ \Om_{k} \big\} \,.
\end{split}
\ee
\par
We are finally in a position to define the energy associated to a deformation $u_\eps\in\A_{\rho,\eps}(\Om_{k\eps})$.
We consider harmonic interactions between nearest neighbours:
\be\label{eng-eps}
\E_{\eps}^\lambda(u_\eps,\rho,k) :=
\tfrac{1}{2}\!\!\! \sum_{\substack{\NN{x}{y}\\x\in \L_{1,\eps}^-(k) \\y\in\L_{\rho,\eps}(k)}}
\left(\Big|\frac{u_\eps(x)-u_\eps(y)}{\e}\Big|-1\right)^2 +
\tfrac{1}{2}\!\!\! \sum_{\substack{\NN{x}{y}\\x\in \L_{\rho,\eps}^+(k)\\y\in\L_{\rho,\eps}(k)}} 
\left(\Big|\frac{u_\eps(x)-u_\eps(y)}{\e}\Big|-\lambda\right)^2 \,,
\ee
where $\lambda\in(0,1)$ is the ratio between the equilibrium lengths of the bonds between atoms in $\L_{\rho,\eps}^+(k)$ and $\L_{1,\eps}^-(k)$, respectively.
For more general assumptions, see Section \ref{sec:rmk}.
Analogously, for $u\in\A_{\rho,\infty}(\Om_{k,\infty})$ we define
\begin{equation}\label{eng-infty}
\E_{\infty}^\lambda(u,\rho,k) := 
\tfrac{1}{2}\!\!\! \sum_{\substack{\NN{x}{y}\\x\in \L_{1,\infty}^-(k) \\y\in\L_{\rho,\infty}(k)}} 
\left(| u(x)-u(y)|-1\right)^2 +
\tfrac{1}{2}\!\!\! \sum_{\substack{\NN{x}{y}\\x\in \L_{\rho,\infty}^+(k) \\y\in\L_{\rho,\infty}(k)}} 
\left(| u(x)-u(y)|-\lambda\right)^2 \,.
\end{equation}
\par
The interaction energy between two nearest-neighbouring atoms $x\in\L_{1,\eps}^-(k)$ and $y\in\L_{\rho,\eps}^+(k)$ 
subject to a deformation $u_\eps$ is then
\be\label{eng-int}
\tfrac{1}{2}\left(\Big|\frac{u_\eps(x)-u_\eps(y)}{\eps}\Big|-1\right)^2+\tfrac{1}{2}\left(\Big|\frac{u_\eps(x)-u_\eps(y)}{\eps}\Big|-\lambda\right)^2 \,,
\ee
with minimum for $\Big|\frac{u_\eps(x)-u_\eps(y)}{\eps}\Big|=\frac{1+\lambda}{2}$.
Remark that the interaction energy of the interfacial bonds is strictly positive. 
Possible generalisations are discussed in Section \ref{sec:rmk}.
\par
\begin{remark}\label{rmk:abc}
We highlight here the difference between the scaling of our functional and the one analysed in \cite{ABC}.
The integral representation derived in that work applies to the limit of  the functionals $\W_{\eps}(u_\eps):=\eps\E_{\eps}^\lambda(u_\eps,1,k)$,
for fixed values of $\lambda$ and $k$ and $\rho=1$, i.e., for a defect-free lattice.
In that approach, the non-interpenetration condition is not assumed.
By computing the energy of a deformation that coincides with the identity on $\L_{1,\eps}^-(k)$
and with a homothety of ratio $\lambda$ on $\L_{\rho,\eps}^+(k)$,
it is immediate to see that $\min\W_\eps=O(\eps)$.
Hence, to characterise the cost of transitions from one equilibrium to the other,
we rescale $\W_\eps$ by a factor $\eps^{-1}$, obtaining $\E_\eps^\lambda(\cdot,1,k)$.
The $\Ga$-limit of $\E_\eps^\lambda(\cdot,1,k)$ can be regarded as the first-order $\Ga$-limit of $\W_\eps$.
\end{remark}
\section{The minimum cost of a deformation}\label{sec:minimum}
In this section we study the minimum cost of the deformations of the rescaled lattice $\L_{\rho,\infty}(k)$. 
\par
Henceforth, the letter $C$ denotes various positive constants whose 
precise value may change from place to place. Its dependence on variables will be 
emphasised only if necessary.
For $N=2,3$, the symbol $\Mnn$ stands for the set of real $N{\times} N$ matrices.
We denote by $GL^+(N)$ the set of matrices with positive determinant
and by $SO(N)$ the set of rotation matrices.
\subsection{Discrete rigidity}
We state a discrete rigidity estimate valid on the triangular cells of the lattice.
This allows one to apply a well-known result of Friesecke, James, and M\"uller \cite{fjm}, which will be employed throughout the paper.
Applications of the rigidity estimate in discrete systems can be found e.g.\ in \cite{BSV,Schm06,Th06}.
\begin{theorem}\label{thm-rigidity}
\cite[Theorem 3.1]{fjm}
Let $N\geq 2$, and let $1\leq p < \infty$.
Suppose that $U\subset\R^{N}$ is a bounded Lipschitz domain. 
Then there exists a constant $C=C(U)$
such that for each $u\in W^{1,p}(U;\R^{N})$ 
there exists a constant matrix $R\in SO(N)$ such that 
\begin{equation}\label{rigidity}
\|\D  u-R\|_{L^{p}(U;\Mnn)} \leq C(U)
\|\dist(\D  u,SO(N))\|_{L^{p}(U)}\,.
\end{equation}
The constant $C(U)$ is invariant under dilation and translation of the domain. 
\end{theorem}
In order to apply Theorem \ref{thm-rigidity} to our discrete setting we will need the following lemma, asserting that the function in
\eqref{cell-energy} is bounded from below by the distance from the set of rotations. 
We adopt here the following notation for the elastic energy
corresponding to an affine deformation $F\in GL^+(2)$ of a cell:
\begin{equation}\label{cell-energy}
\E(F):= \sum_{i=1}^3( |F\w_i|-1)^2 \,.
\end{equation}
\begin{lemma}\label{lemma-equiv}
There exists $C>0$ such that 
\begin{equation}\label{bound-from-below}
\dist^2(F,SO(2)) \leq  C \E(F)  \quad \text{for every} \ F\in GL^+(2) \,. 
\end{equation}
\end{lemma}
\begin{proof}
Set $\delta_i:= |F \w_i| -1$ and $\delta:=(\delta_1,\delta_2,\delta_3)$, then $\E(F) = \sum_{i=1}^3 \delta_i^2=|\delta|^2$.
Without loss of generality we may assume that $F \w_1 =(1+\delta_1)\w_1$ as in Figure~\ref{fig-triangle}.  
\begin{figure}[b]
\centering
\psfrag{w1}{$\w_1$}
\psfrag{w2}{$\w_2$}
\psfrag{d1}{$\delta_1$}
\psfrag{t}{$\theta$}
\psfrag{f2}{$F\w_2$}
\includegraphics[scale=1.2]{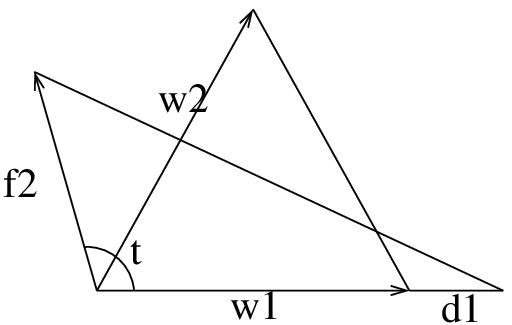}
\caption{Notation for Lemma \ref{lemma-equiv}.}
\label{fig-triangle}
\end{figure}
We have
\begin{equation}\label{dist-est-2}
\dist^2(F,SO(2)) \leq |F-I|^2 \leq C \big( | (F-I) \w_1 |^2 +   |(F-I)\w_2 |^2  \big) 
=C \big(\delta_1^2  + | (F-I)\w_2 |^2  \big)
\end{equation}
and
\be\label{dist-estimate}
|(F-I)\w_2 |^2 = 1 + (1+\delta_2)^2 -2(1+\delta_2)\cos\big(\theta - \tfrac{\pi}{3}\big) \,,
\ee
where $\theta$ is the angle (measured anticlockwise) between $\w_1$ and $F\w_2$,
which is determined by
\be\label{dist-estimate2}
\cos\theta = 
\frac{(1+\delta_1)^2  +  (1+\delta_2)^2 - (1+\delta_3)^2 }{ 2 (1+\delta_1)(1+\delta_2)}
\quad \text{and} \quad \sin\theta >0 \,.
\ee
Remark that the condition $\sin\theta>0$ follows from the assumption $F\in GL^+(2)$.
From \eqref{dist-est-2} and \eqref{dist-estimate} we deduce that 
\begin{equation}\label{big-delta}
\dist^2(F,SO(2)) \leq C(1+ \delta_1^2 + \delta_2^2) \,.
\end{equation}
On the other hand, by computing the second order Taylor expansion of \eqref{dist-estimate} about the point $\delta=(0,0,0)$ and taking into account \eqref{dist-estimate2} we see that 
\begin{equation}\label{small-delta}
\dist^2(F,SO(2)) \leq C |\delta|^2 + o(|\delta|^2) \,,  
\end{equation}
all first derivatives being zero at $(0,0,0)$.
Then \eqref{bound-from-below} readily follows from \eqref{small-delta} for $\E(F)$ small,
from \eqref{big-delta} for values of $\E(F)$ larger than one,
and by continuity in the intermediate case.
\end{proof}
\begin{remark}\label{rmk-equiv}
Arguing as above, one sees that
$$
\dist^2(F,\lambda \, SO(2)) \leq  C \sum_{i=1}^3( |F\w_i|-\lambda)^2
$$
for every $F\in GL^+(2)$.
\end{remark}
\subsection{Estimates on the cost of defect-free and dislocated deformations} \label{sec:gamma}
We now introduce the minimum cost of a deformation of the rescaled lattice $\L_{\rho,\infty}(k)$.
To this end, we consider deformations $v\in\A_{\rho,\infty}(\Om_{k,\infty})$ that are in equilibrium away from the interface,
i.e., such deformations $v$ bridge the two wells $SO(2)$ and $\frac\lambda\rho SO(2)$ of the lattice energy \eqref{eng-infty} around the interface.
Because of the rotational invariance, we may assume that, for some $M>0$ and $R\in SO(2)$, $\D v=I$ if $x_1\in(-\infty,-M)$
and $\D v=\frac\lambda\rho R$ if $x_1\in(M,+\infty)$. We recall that the admissible deformations are piecewise affine
on the triangulation determined by $\L_{\rho,\infty}(k)$.
\par
Therefore, given $R\in SO(2)$, $\rho\in(0,1]$, and $k\in\N$, we define the 
minimum cost of a transition from one equilibrium to the other as
\begin{equation}\label{gammah}
\begin{split}
\ga^\lambda(\rho,k,R):=\inf\big\{ & \E_{\infty}^\lambda(v,\rho,k) \colon M>0\,,
\  v\in \A_{\rho,\infty}(\Om_{k,\infty})\,, \\
& \D v=I \ \text{for} \ x_1\in(-\infty,-M)\,,\, \ \D v=\tfrac\lambda\rho R \ \text{for} \ x_1\in(M,+\infty) 
\big\}\,.
\end{split}
\end{equation}
In Section \ref{sec:lim} we will show that $\ga^\lambda(\rho,k,R)$ enters in the $\Ga$-limit of the total interaction energy
as the atomic distance tends to zero. Here we investigate the behaviour of the minimal cost
for defect-free and dislocated configurations.
\par
In the following proposition we prove that \eqref{gammah} is actually independent of the choice of the rotation $R$.
Hence, we will write
$$
\ga^\lambda(\rho,k):=\ga^\lambda(\rho,k,I)\,.
$$
This means that the growth direction of the nanowire is not captured at the scaling of the functional considered here.
Moreover, the following proof shows that the specimen is not sensitive to bending in our model.
\begin{proposition}\label{invariance} 
For every $R\in SO(2)$ we have 
$$
\ga^\lambda(\rho,k,R)=\ga^\lambda(\rho,k,I)\,.
$$
\end{proposition}
\begin{proof}
Let $R,Q\in SO(2)$ and $M\in\N$. 
Suppose that $v\in\A_{\rho,\infty}(\Om_{k,\infty})$ is such that $\D v=I$ if $x_1\in(-\infty,-M)$
and $\D v=\frac\lambda\rho R$ if $x_1\in(M,+\infty)$.
We show that there exists a sequence $\{v_n\}\subset\A_{\rho,\infty}(\Om_{k,\infty})$ such that
$v_n=v$ for $x_1\in(-\infty,M)$, 
$\D v_n=\frac\lambda\rho Q$ for $x_1\in (M{+}n{+}1,+\infty)$, and 
\begin{equation}\label{dis-invariance}
\inf_n \E_{\infty}^\lambda(v_n,\rho,k)\leq  \E_{\infty}^\lambda(v,\rho,k) \,.
\end{equation}
This in turn yields 
$$
\ga^\lambda(\rho,k,Q) \leq \ga^\lambda(\rho,k,R) \,,
$$
which concludes the proof by exchanging the role of $R$ and $Q$. Next we prove the existence of such a sequence $\{v_n\}$.
\par
Let $\theta_1,\theta_2 \in [0, 2\pi)$ be such that
\begin{equation*}
R=
\begin{pmatrix}
\cos\theta_1  & -\sin\theta_1 \\
\sin\theta_1 & \cos\theta_1 
\end{pmatrix}\,,
\quad
Q=
\begin{pmatrix}
\cos\theta_2  & -\sin\theta_2 \\
\sin\theta_2 & \cos\theta_2 
\end{pmatrix} \,.
\end{equation*}
For any $n>1$, consider the smooth path $R_n\colon\R\to SO(2)$ connecting $R$ and $Q$ defined for $s\in[M{+}1,M{+}n]$ by
\begin{equation*}
R_n(s) := 
\begin{pmatrix}
\cos\big(\theta_1+(\theta_2{-}\theta_1)\frac{s-M-1}{n-1}\big) & -\sin\big(\theta_1+(\theta_2{-}\theta_1)\frac{s-M-1}{n-1}\big) \\[.3em]
\sin\big(\theta_1+(\theta_2{-}\theta_1)\frac{s-M-1}{n-1}\big) & \cos\big(\theta_1+(\theta_2{-}\theta_1)\frac{s-M-1}{n-1}\big)
\end{pmatrix} \,.
\end{equation*}
The definition is completed by setting $R_n(s):=R$ for $s< M{+}1$, $R_n(s):=Q$ for $s>M{+}n$.
Let $u_n$ be a piecewise affine function in $\Om_{k,\infty}$ (with respect to the triangulation $\TT_\rho$)
such that
\begin{equation*}
u_n (x) = \integlin{M{+}1}{x_1}{R_n(s) \begin{pmatrix}1 \cr 0 \end{pmatrix} }{s} + R_n(x_1)\begin{pmatrix}0 \cr x_2 \end{pmatrix} 
\quad \text{for}\ x\in \L_{\rho,\infty}(k) \,.
\end{equation*}
This in particular implies, because of the piecewise affine structure of $u_n$,
\beas
\D u_n = R && \text{in}\ (-\infty,M{+}\tfrac12)\times(0,\tfrac{\sqrt3}{2}k) \,, \\
\D u_n = Q && \text{in}\ (M{+}n{+}\tfrac12,+\infty)\times(0,\tfrac{\sqrt3}{2}k) \,.
\eeas
Notice that 
\be\label{1655}
\tfrac\lambda\rho\D u_n=\D v=\tfrac\lambda\rho R
\ee
in a zigzag-shaped set containing $(M,M{+}\tfrac12)\times(0,\tfrac{\sqrt3}{2}k)$.
Now let $v_n$ be defined by
\begin{equation*}
v_n(x ):=
\begin{cases}
v(x) & \text{in}\ (-\infty,M]\times(0,\tfrac{\sqrt3}{2}k) \,, \\
\frac\lambda\rho u_n(x) + c_n & \text{in}\ (M,+\infty)\times(0,\frac{\sqrt3}{2}k)
\,,
\end{cases}
\end{equation*}
where $c_n$ is chosen in such a way that $v_n$ is continuous. Next remark that for
$x=(x_1,x_2)\in\L_{\rho,\infty}(k)\cap\Big((M{+}\frac12,M{+}n{+}\tfrac12)\times(0,\tfrac{\sqrt3}{2}k)\Big)$ we have
\bea
\label{130417-1} u_n(x_1{+}1,x_2)-u_n(x_1,x_2)&=& \begin{pmatrix}
\cos\big(\theta_1+(\theta_2{-}\theta_1)\frac{x_1-M-1}{n-1}\big) \\[.3em]
\sin\big(\theta_1+(\theta_2{-}\theta_1)\frac{x_1-M-1}{n-1}\big)
\end{pmatrix}+\Phi_n(x) \,,\\
\label{130417-2}\quad u_n(x_1{+}\tfrac12,x_2{+}\tfrac{\sqrt3}2)-u_n(x_1,x_2) &=& \begin{pmatrix}
\cos\big(\theta_1+(\theta_2{-}\theta_1)\frac{x_1-M-1}{n-1}+\frac\pi3\big) \\[.3em]
\sin\big(\theta_1+(\theta_2{-}\theta_1)\frac{x_1-M-1}{n-1}+\frac\pi3\big)
\end{pmatrix}+\Psi_n(x) \,,
\eea
where $|\Phi_n(x)|\le C/n$ and $|\Psi_n(x)|\le C/n$ uniformly in $x$.
This shows that $\det\D u_n>0$ a.e.\ in $\Om_{k,\infty}$ for $n$ large, hence $v_n\in\A_{\rho,\infty}(\Om_{k,\infty})$.
Taking into account \eqref{1655}, it follows that $v_n\in\A_{\rho,\infty}(\Om_{k,\infty})$.
Moreover
$$
\E_{\infty}^\lambda(v_n,\rho,k) \leq \E_{\infty}^\lambda(v,\rho,k) + \E_{\infty}^\lambda(\tfrac\lambda\rho u_n(x) {+} c_n,\rho,k) \leq \E_{\infty}^\lambda(v,\rho,k) + O(n^{-1}) \,,
$$
which yields \eqref{dis-invariance} on letting $n\to \infty$.
In the last inequality we used the fact that the number of triangles contained in the set $(M{+}\frac12,M{+}n{+}\tfrac12)\times(0,\tfrac{\sqrt3}{2}k)$
is of order $n$ and in each of them the contribution to the total energy is of order $n^{-2}$,
by \eqref{130417-1}--\eqref{130417-2} and because the energy is quadratic. 
\end{proof}
We now prove estimates on the asymptotic behavior of $\ga^\lambda(\rho,k)$ as $k\to +\infty$.
We consider two main cases: $\rho=1$ (defect-free) and $\rho=\lambda$ (which gives a possible configuration with dislocations).
It turns out that the growth of $\ga^\lambda(\rho,k)$ is quadratic in $k$ if $\rho=1$
and linear in $k$ if $\rho=\lambda$.
This shows that the dislocations are favoured when the number $k$ of lines in the substrate is sufficiently large.
\begin{proposition}[Estimate in the defect-free case, $\rho=1$]\label{lowerbound}
There exist $C_1,C_2>0$ such that for every $k\in\N$
\begin{equation*}
C_1 k^2 \leq \ga^\lambda(1,k) \leq C_2 k^2 \,.
\end{equation*}
\end{proposition}
\begin{proof}
The upper bound is proven by comparing test functions for $\ga^\lambda(1,k)$ with those for 
$\ga^\lambda(1,1)$.
Let $v\in\A_{1,\infty}(\Om_{1,\infty})$ and define $u\in \A_{1,\infty}(\Om_{k,\infty})$ by $u(x) := k v(x/k)$.
Assume that $x=\xi_1\w_1+\xi_2\w_2$ and $y=(\xi_1+1)\w_1+\xi_2\w_2$, for some $\xi_1\in\Z$ and $\xi_2\in\Z\cap[0,\frac{\sqrt3}{2}k]$;
notice that $x,y\in\L_{1,\infty}(k)$ and $y-x=\w_1$.
Let $m=\lfloor\xi_1/k\rfloor$; then it is readily seen that 
$u(x)-u(y)=\frac1k(u(mk,0)-u((m+1)k,0)=v(m,0)-v(m+1,0)$ if $\xi_1+\xi_2<(m+1)k$ and $u(x)-u(y)=v(m,1)-v(m+1,1)$ otherwise.
A similar relationship holds for nearest neighbours of the type $y-x=\w_2$ and $y-x=\w_3$.
This shows that $\ga^\lambda(1,k)\leq \E_{\infty}^\lambda(u,1,k)\leq C \, \E_{\infty}^\lambda(v,1,1) \,k^2$, which yields 
$\ga^\lambda(1,k) \leq C \, \ga^\lambda(1,1)\,k^2$.
\par
Next we prove the lower bound.
On the contrary, suppose that there exist a sequence $k_j\nearrow\infty$ and a sequence 
$\{u_j\}\subset \A_{1,\infty}(\Om_{k_j,\infty})$ such that 
\begin{equation}\label{infinitesimal}
\frac{1}{k_j^2}\,\E_{\infty}^\lambda(u_j,1,k_j) \to 0 \,. 
\end{equation}
Define $v_j\colon\Om_{1,\infty}\to \R^2$ as $v_j(x) := \frac{1}{k_j} u_j(k_j x)$. 
Accordingly, we consider the rescaled lattices $\L_j:=\frac1{k_j}\L_{1,\infty}(k)$
and $\L_j^\pm:=\frac1{k_j}\L_{1,\infty}^\pm(k)$;
here two points $x,y$ are said to be nearest neighbours ($\NN{x}{y}$) if $k_jx$, $k_jy$ fulfill the corresponding property in the lattice $\LL_\rho$.
Notice that
\begin{equation*}
\E_{\infty}^\lambda(u_j,1,k_j) = \tfrac{1}{2}
\!\!\!\sum_{\substack{\NN{x}{y}\\x\in\L_j^-,\ y\in\L_j}} \!\!
\left(\frac{| v_j(x)-v_j(y)|}{\frac1{k_j}}-1\right)^2 +
\tfrac{1}{2}
\!\!\!\sum_{\substack{\NN{x}{y}\\x\in\L_j^+,\ y\in\L_j}} \!\!
\left(\frac{| v_j(x)-v_j(y)|}{\frac1{k_j}}-\lambda\right)^2 \,,
\end{equation*}
so this term controls the (piecewise constant) gradient of $v_j$.
Set 
\beas
\omega_j^-&:=& \{\xi_1 \w_1 + \xi_2 \w_2 \colon \xi_1\in(-1,-\tfrac1{k_j})\,,\ \xi_2\in(0,1) \} \,, \\ 
\omega^-&:=& \{\xi_1 \w_1 + \xi_2 \w_2 \colon \xi_1\in(-1,0)\,,\ \xi_2\in(0,1) \} \,, \\ 
\omega^+&:=& \{\xi_1 \w_1 + \xi_2 \w_2 \colon \xi_1\in(0,1)\,,\ \xi_2\in(0,1) \} \,, \\
\omega&:=& \{\xi_1 \w_1 + \xi_2 \w_2 \colon \xi_1\in(-1,1)\,,\ \xi_2\in(0,1) \} \,.
\eeas
By \eqref{infinitesimal} there exists $v\in W^{1,2}(\om;\R^2)$ such that $\D v_j\wto \D v$ up to subsequences.
We now apply the estimate from Lemma \ref{lemma-equiv} to each of the triangles of $\frac{1}{k_j}\TT_1$ 
contained in $\omega_j^-$. After integration, \eqref{bound-from-below} gives 
(recall that any triangle has area $C/k_j^2$)
\bes
\begin{split}
& \integ{\omega_j^-}{\dist^2(\D v_j, SO(2))}{x} =
\frac{C}{k_j^2} \sum_{\substack{T\in\frac{1}{k_j}\TT_1\\T\subset\om_j^-}}\dist^2(\D v_j, SO(2))\\
\le\ & \frac{C}{k_j^2} \sum_{\substack{T\in\frac{1}{k_j}\TT_1\\T\subset\om_j^-}}\sum_{i=1}^3 (|\D v_j\cdot\w_i|-1)^2 
\leq \frac{C}{k_j^2} \, \E_{\infty}^\lambda(u_j,1,k_j) \to0 \,,
\end{split}
\ees
where the convergence to zero follows from \eqref{infinitesimal}.
Therefore, Theorem \ref{thm-rigidity} ensures that there exists a sequence 
$\{R_j^-\}\subset SO(2)$ such that 
$$
\integ{\omega_j^-}{|\D v_j  - R_j^-|^2 }{x} \to 0 \,.
$$
Arguing in a similar way for the set $\omega^+$ and using Remark \ref{rmk-equiv}, we deduce that there exists $\{R_j^+\}\subset SO(2)$ 
such that 
$$
\integ{\omega^+}{|\D v_j  - \lambda R_j^+|^2}{x} \to 0 \,.
$$
Up to extracting a subsequence such that $R_j^-\to R^-$ and $R_j^+\to R^+$ for some $R^+,R^-\in SO(2)$, we obtain
$\D v=R^-$ in $\om^-$ and $\D v=\lambda R^+$ in $\om^+$.
Since $v\in W^{1,2}(\om;\R^2)$, we conclude that
$\rank(R^-{-}\lambda R^+)\le1$.
This implies in particular that $\lambda=1$,
which is a contradiction to $\lambda\in(0,1)$.
Hence the lower bound follows.
\end{proof}
\begin{remark} \label{rmk:lowerbound}
Arguing as before, one can show that $C_1 k^2 \leq \ga^\lambda(\rho,k) \leq C_2 k^2$ for every $\rho\neq\lambda$,
where the constants $C_1,C_2$ are uniform in $k$ but may depend on $\rho$.
In the remaining case $\rho=\lambda$ the growth of the minimal energy is linear in $k$, as shown in the next proposition.
\end{remark}
\begin{proposition}[Estimate for $\rho=\lambda$] \label{prop:lambda}
There exist positive constants $C_1',C_2'$ such that for every $k$
\bes
C_1' k\le \gamma^\lambda(\lambda,k)\le C_2' k \,.
\ees
\end{proposition}
\begin{proof}
The lower bound follows from the fact that the minimum cost of the interfacial bonds is strictly positive and such bonds are at least $2k+1$. For the upper bound, we consider
the identical deformation, which is in equilibrium except for the interfacial bonds.
Their cost is bounded by the total number of lines, the maximal number of neighbours,
and the maximal length of a bond, which can be estimated as in Remarks~\ref{rmk:Del1} and~\ref{rmk:Del2}. 
\end{proof}
From Proposition \ref{lowerbound} and Proposition \ref{prop:lambda} we obtain the following consequence.
\begin{corollary}\label{coro}
The following inequality holds:
\begin{equation}\label{energeticineq} 
\gamma^\lambda(1,k)
>
\inf_{\rho}\gamma^\lambda(\rho,k) 
\end{equation}
for $k$ sufficiently large.
\end{corollary}
\begin{remark}
One can exhibit different configurations (corresponding to different values of $\rho$), for which the energy grows slower than quadratic.
For example, for $\rho_k(\lambda):=\frac{\lambda}{1-k^{\alpha-1}}$ with $0\le\alpha<\frac12$
(notice that $\rho_k(\lambda)\to\lambda$ as $k\to\infty$),
one can define the following deformation of $\L_{\rho_k(\lambda),\infty}(k)$:
$$
u(x):=\begin{cases}x &\text{if}\ x\in\L_{1,\infty}^-(k)\,,\\\frac\lambda{\rho_k(\lambda)}x &\text{if}\ x\in\L_{\rho_k(\lambda),\infty}^+(k)\,,\end{cases}
$$
extended to $\Om_{k,\infty}$ by piecewise affine interpolation and still denoted by $u\in\A_{\rho_k(\lambda),\infty}(\Om_{k,\infty})$.
A direct computation shows that for $k>1$
$$
\E^\lambda_\infty(u,\rho_k(\lambda),k)\le C k^{1+2\alpha} < C k^2 \,,
$$
because of the choice of $\alpha$.
\end{remark}
An interesting issue for the applications is to find the largest value $k_c$ such that the defect-free model is energetically convenient for $k\le k_c$;
we leave this for future research. Here we just observe  that \eqref{energeticineq} does not hold for $k=1$ as stated in the following remark.
\begin{remark}
Fix $\lambda\in(0,1)$. A straightforward computation shows that $\ga^\lambda(1,1)=\frac{3}{4}(1-\lambda)^2=\gamma^\lambda(\rho,1)$ for every $\rho>\frac12$.
On the other hand, for $\rho\le\frac12$, which corresponds to introducing dislocations, $\gamma^\lambda(\rho,1)\ge(1-\lambda)^2$.
In this case, \eqref{energeticineq} does not hold.
\end{remark}
For further comments about the estimates on the cost of the deformations,
see Section~\ref{rmk:non-int}.
\section{The limiting variational problem}\label{sec:lim}
For fixed $\rho\in(0,1]$ and $k\in\N$ we address the question of the $\Gamma$-convergence of the sequence of functionals 
$\{\I_\e\}$ defined by
$$
\I_{\e}(\tilde u_\e) := \E^\lambda_\e(u_\e,\rho,k) \quad \text{for}\ \ut_\e \in \widetilde\A_{\rho,\e}(\Om_{k}) \,,
$$
where  $u_\e(x):=\ut_\e (A_\e^{-1}x)\in\A_{\rho,\e}(\Om_{k\eps})$, see \eqref{tildeA}
and recall that
$$
\dsp A_\e:= 
\begin{pmatrix}
1 & \frac{\e-1}{\sqrt{3}}
\\
0 & \e
\end{pmatrix}
\quad \text{and thus} \quad
\dsp A_\e^{-1}:= 
\begin{pmatrix}
1 & \frac{1-\e}{\eps\sqrt{3}}
\\[.5em]
0 & \frac1\e
\end{pmatrix} \,.
$$
From now on we will drop the dependence of $\E^\lambda_\e(\cdot,\rho,k)$ and $\E^\lambda_\infty(\cdot,\rho,k)$
on $\lambda$, $\rho$, and $k$, 
since these quantities will be fixed during the whole proof of the $\Ga$-convergence.
\par
In the proof we use methods that are  common in dimension reduction problems,
see e.g.\ \cite{fjm,mm,mp}, and
we adapt them to the discrete setting.
The symbol $\co(SO(2))$ stands for the convex hull of $SO(2)$ in $\Mdd$.
\subsection{Compactness and lower bound}\label{section:four}
\begin{theorem}\label{thm1}
Let $\{\ut_\e\}\subset \widetilde\A_{\rho,\eps}(\Om_{k})$ be a 
sequence such that 
\begin{equation} \label{equibounded}
\limsup_{\e\to 0^{+}} \I_{\e}(\ut_\e) \leq C \,.
\end{equation}
Then there exists a subsequence (not relabeled) such that 
\begin{equation*}
\D \ut_\e A_\e^{-1}\weakst (\partial_1 \ut \, | \, d_2) \quad \text{weakly* in}\ L^{\infty}(\Om_k;\Mdd)\,,
\end{equation*}
where the functions $\ut \in W^{1,\infty}(\Om_k;\R^2)$, $d_2\in L^{\infty}(\Om_k;\R^2)$
are independent of $\w_2$, i.e., 
$\partial_{\w_2}\ut = \partial_{\w_2} d_2 = 0$.
Moreover,
\begin{equation} \label{charact}
(\partial_1 \ut \, | \, d_2)\in 
\begin{cases}
\co(SO(2))     &  \text{a.e.\ in}\ \Om_k^-\,,\\
\co(\frac\lambda\rho SO(2))   &  \text{a.e.\ in}\ \Om_k^+\,. 
\end{cases}
\end{equation}
Finally, for each such subsequence we have
\begin{equation} \label{gammaliminf}
\liminf_{\e\to 0^{+}} \I_{\e}(\tilde u_\e)\geq\ga^\lambda(\rho,k) \,.
\end{equation}
\end{theorem}
\begin{proof}
The assumption \eqref{equibounded} implies that $\|\D u_\e\|_{L^{\infty}(\Om_{k\eps};\Mdd)}$ is uniformly  bounded in $\eps$
and so $\{\D \ut_\e A_\e^{-1}\}$ is a uniformly bounded sequence in $L^{\infty}(\Om_k;\Mdd)$. Moreover,
$$ 
\D \ut_\e A_\e^{-1} = \big( \partial_1\ut_\e \, \big|  \, \tfrac{1}{\e}\partial_{\w_\e}\ut_\e \big) \quad \text{where}\ 
\w_\e:= \big( \tfrac{1-\e}{\sqrt{3}}, 1 \big) \,.
$$ 
Since $\w_\e \to \frac{2}{\sqrt{3}} \w_2$, there exist a subsequence
and functions $\ut\in  W^{1,\infty}(\Om_k;\R^2)$ and $d_2 \in L^{\infty}(\Om_k;\R^2) $ such that
$\partial_1\ut_\e \weakst \partial_1 \ut$ weakly* in $L^{\infty}(\Om_k;\Mdd)$,
$\ut\in  W^{1,\infty}(\Om_k;\R^2)$ is independent of $\w_2$, 
and $ \frac{1}{\e}\partial_{\w_\e}\ut_\e \weakst d_2$.
In order to show $\partial_{\w_2}d_2=0$ and \eqref{charact}, we  apply the rigidity estimate \eqref{rigidity} to the sequence $u_\e$.
To this aim, we divide the domain $\Om_{k\eps}$ into subdomains $\subdo_\e^i$ of size $\eps$ 
and consider those sets $\subdo_\e^i$ where the energy is small.  
More precisely, for $i\in\Z$ set 
\be
\label{defin:subdo}
\subdo_\eps^i:=
\begin{cases}
\{ \xi_1\w_1+\xi_2\w_2\colon \xi_1\in(i\eps,(i{+}1)\eps)\,,\ \xi_2\in(0, k\eps) \} &\text{if}\ i\le-1 \,,\\
\{ \xi_1\w_1+\xi_2\w_2\colon \xi_1\in(i\rho\eps,(i{+}1)\rho\eps)\,,\ \xi_2\in(0, k\eps) \} & \text{if}\ i\ge0 \,.
\end{cases}
\ee
Let $I_\eps:=\{i\in\Z\colon i\neq-1,\ \subdo_\eps^i\subset\Om_{k\eps}\}$,
see Figure~\ref{parall1}.
\begin{figure}
\centering
\psfrag{4}{\dots}
\psfrag{-5}{\hspace{.3cm}\dots}
\psfrag{-4}{$\subdo_\e^{-4}$}
\psfrag{-3}{$\subdo_\e^{-3}$}
\psfrag{-2}{$\subdo_\e^{-2}$}
\psfrag{-1}{$\subdo_\e^{-1}$}
\psfrag{0}{$\subdo_\e^{0}$}
\psfrag{1}{$\subdo_\e^{1}$}
\psfrag{2}{$\subdo_\e^{2}$}
\psfrag{3}{$\subdo_\e^{3}$}
\psfrag{p-1}{\hspace{-.2cm}$-\eps$}
\psfrag{p0}{\hspace{.02cm}$0$}
\psfrag{p1}{\hspace{-.1cm}$\rho\eps$}
\includegraphics[width=.8\textwidth]{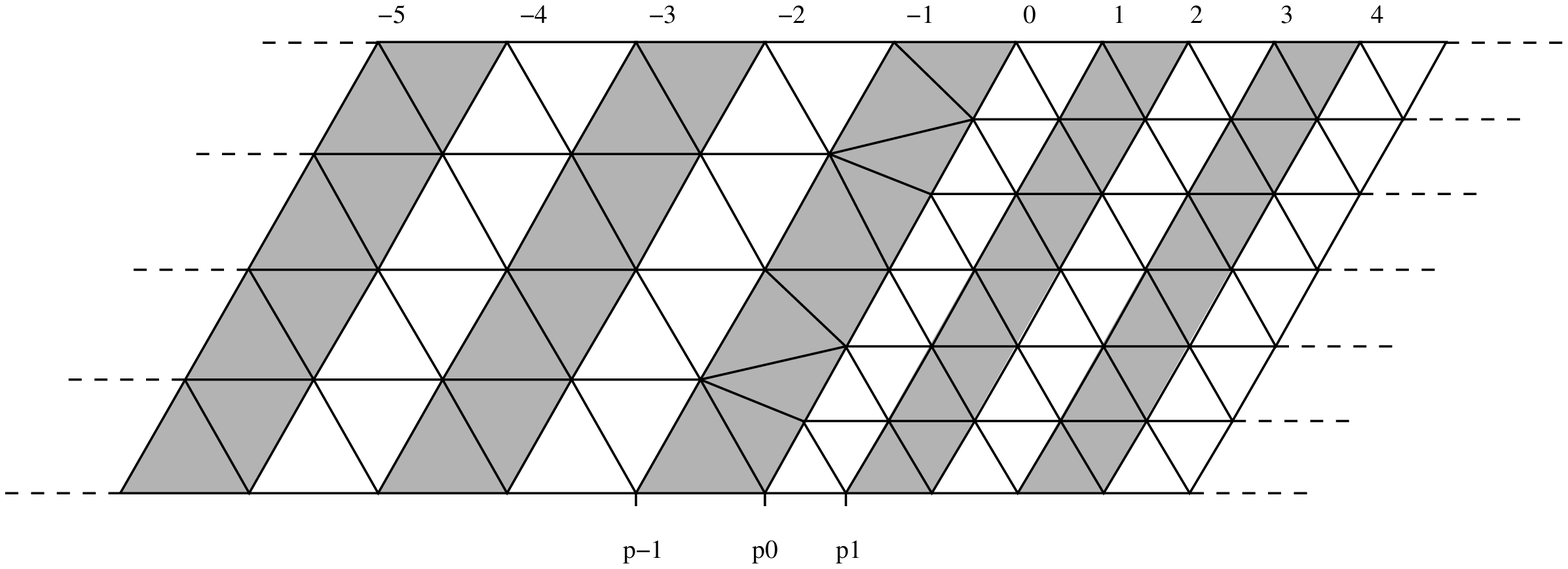}
\caption{The parallelograms $\subdo_\e^i$.}
\label{parall1}
\end{figure}
Let $\delta \in(0,1)$ and set 
$$
\subdo_\e :=\bigcup_{i\in I_\eps} \{\subdo_\e^i \colon \E^i_\e(u_\e) < \delta \} \,, 
\quad
\widetilde\subdo_\e := \Om_{k\eps} \setminus \subdo_\e \,.
$$
Here $\E^i_\e(u_\e)$ denotes the contribution of $\overline{\subdo_\e^i}$ to $\E_\e(u_\e)$,
i.e., in the sum \eqref{eng-eps} one considers only the bonds contained in $\overline{\subdo_\e^i}$. 
Notice that in the previous definition, we let the subdomain $\subdo_\e^{-1}$
(corresponding to the interface) be contained in $\widetilde\subdo_\e$ for convenience.
The assumption \eqref{equibounded} yields for $\delta$ sufficiently small
\begin{equation}\label{cardinality}
\#\{i\in \Z\colon \subdo_\e^i \cap\widetilde\subdo_\e \neq \emptyset\} \leq \frac{C}{\delta} \,.
\end{equation}
Thanks to the definition of $\subdo_\eps$, we can apply Lemma \ref{lemma-equiv} and Remark \ref{rmk-equiv} to deduce that 
\bea\label{small-dist}
\dist^2(\D u_\e(x),SO(2)) &<& C\delta \quad \text{for a.e.}\ x\in\subdo_\e \cap \Om_{k\eps}^-\,,\\
\label{small-dist2}
\dist^2(\D u_\e(x),\tfrac\lambda\rho SO(2))& <& C\delta \quad \text{for a.e.}\ x\in\subdo_\e \cap \Om_{k\eps}^+ \,. 
\eea
Next by the rigidity estimate \eqref{rigidity} there exists $R^i_\e\in SO(2)$ such that 
\bea\label{dist-rot}
\integ{\subdo_\e^i}{ |\D u_\e - R_\e^i |^2 }{x} &\leq& C \integ{\subdo_\e^i}{ \dist^2(\D u_\e, SO(2)) }{x}
      \quad\text{if}\ \subdo_\e^i \subset \Om_{k\eps}^-  \,,\\
\label{dist-rot2}
\integ{\subdo_\e^i}{ |\D u_\e - \tfrac\lambda\rho R_\e^i |^2 }{x} &\leq& C \integ{\subdo_\e^i}{ \dist^2(\D u_\e, \tfrac\lambda\rho SO(2)) }{x} 
       \quad\text{if}\ \subdo_\e^i \subset \Om_{k\eps}^+  \,,
\eea
with $C$ independent of $i$ and $\eps$ since the shape of the domains is the same.
Let $R_\e \colon \Om_{k\eps} \to SO(2)$ be such that $R_\e(x)=R_\e^i$ for $x\in \subdo_\e^i$, $i\in\Z$; notice that $R_\e$ is independent of $\w_2$.
By \eqref{small-dist}--\eqref{dist-rot2} and
the uniform bound on $\|\D u_\e\|_{L^{\infty}(\Om_{k\eps};\Mdd)}$ together with \eqref{cardinality},
we get
\bes
\begin{split}
& \integ{\Om_{k\eps}^-}{ |\D u_\e - R_\e|^2 }{x}  +  
\integ{\Om_{k\eps}^+}{ |\D u_\e - \tfrac\lambda\rho R_\e|^2 }{x}  \\
=\ & \integ{\Om_{k\eps}^-\cap\subdo_\e}{ |\D u_\e - R_\e|^2 }{x}  +  
\integ{\Om_{k\eps}^+\cap\subdo_\e}{ |\D u_\e - \tfrac\lambda\rho R_\e|^2  }{x}   \\
&+  
\integ{\Om_{k\eps}^-\setminus\subdo_\e}{ |\D u_\e - R_\e|^2 }{x}   +  
\integ{\Om_{k\eps}^+\setminus\subdo_\e}{ |\D u_\e - \tfrac\lambda\rho R_\e|^2  }{x}  \\
\leq\ &   C\e \delta  +  \frac{C}{\delta}\e^2 \,.
\end{split}
\ees
Therefore, by a change of variables,
\bes
\int_{\Om_k^-} |\D \ut_\e A_\e^{-1}- R_\e|^2  dx   +  
\int_{\Om_k^+} |\D \ut_\e A_\e^{-1} - \tfrac\lambda\rho R_\e|^2  dx
\leq C \delta  +  \frac{C}{\delta}\e \,,
\ees
where $R_\e$ has been extended to $\Om_k$ in such a way that the independence on $\w_2$ is preserved.
Since $R_\e\in SO(2)$, its weak* limit does not 
depend on $\w_2$ and takes values in $\co(SO(2))$. 
Choosing $\delta = \sqrt{\e}$ and letting $\e \to 0^+$ we obtain \eqref{charact}.
\par
In order to show \eqref{gammaliminf} we define $ v_\e(x) :=\frac 1 \e u_\e (\e x) $ and observe that 
by assumption \eqref{equibounded}
\begin{equation}\label{rescaled-bound}
\E_\e(u_\e)=\sum_{\substack{\NN{x}{y}\\x,y\in\frac{1}{\e}\Om_{k\eps}\\x\in \L_{1,\infty}^-(k) \\y\in\L_{\rho,\infty}(k)}}
\big( | v_\e(x)-v_\e(y)|-1 \big)^2 +
\sum_{\substack{\NN{x}{y}\\x,y\in\frac{1}{\e}\Om_{k\eps}\\x\in \L_{\rho,\infty}^+(k) \\y\in\L_{\rho,\infty}(k)}}
\big( | v_\e(x)-v_\e(y)|-\lambda \big)^2 
\leq C \,.
\end{equation}
Then there exist two sub-parallelograms, see \eqref{defin:subdo},
$$
\Subdo_\e^-:= \tfrac1\eps \subdo_\e^j \,, \quad \Subdo_\e^+:= \tfrac1\eps \subdo_\e^{j'} \,, 
$$
with $j,j'\in I_\eps$, $j<-1$, and $j'\ge0$,
such that the contribution to \eqref{rescaled-bound} of the bonds in  
$\overline{\Subdo_\e^-}$ and $\overline{\Subdo_\e^+}$ is less than $C\e/L$.
Hence, by Theorem~\ref{thm-rigidity}, Lemma~\ref{lemma-equiv}, and Remark~\ref{rmk-equiv},
arguing as in the proof of Proposition \ref{lowerbound} we deduce the existence of $R_\e^-,R_\e^+\in SO(2)$ with
\bes
\begin{split}
& \integ{\Subdo_\e^-}{\big| \D v_\e -R_\e^- \big|^2}{x} 
 + \integ{\Subdo_\e^+}{\big| \D v_\e -\tfrac\lambda\rho R_\e^+ \big|^2}{x} \\
\le\ & C \integ{\Subdo_\e^-}{\dist^2(\D v_\e,SO(2))}{x} 
 + C \integ{\Subdo_\e^+}{\dist^2(\D v_\e,\tfrac\lambda\rho SO(2))}{x}
\leq   C \e \,.
\end{split}
\ees
From the Poincar\'e inequality there exist $c_\e^-, c_\e^+\in \R^2$ such that
\begin{equation}\label{eq-small-energy}
\begin{split}
&\integ{\Subdo_\e^-}{\Big(\big| v_\e -(R_\e^- x + c_\e^-) \big|^2 + \big| \D v_\e -R_\e^- \big|^2\Big)}{x} \\
& + \integ{\Subdo_\e^+}{\Big(\big| v_\e -(\tfrac\lambda\rho R_\e^+ x + c_\e^+) \big|^2 + \big| \D v_\e -\tfrac\lambda\rho R_\e^+ \big|^2\Big)}{x}
\leq   C \e \,.
\end{split}
\end{equation}
Let $\bar v_\e\colon\Om_{k,\infty}\to\R^2$ be the piecewise affine  deformation (with respect to the triangulation $\TT_\rho$) such that
$$
\bar v_\e(x) :=
\begin{cases}
 R_\e^- x + c_\e^- & \text{for}\ x=\xi_1 \w_1 + \xi_2 \w_2 \colon \xi_1\le j \,,\\
 v_\e(x)        &  \text{for}\  x=\xi_1 \w_1 + \xi_2 \w_2 \colon j{+}1\le\xi_1\le j' \,,\\
\frac\lambda\rho R_\e^+ x + c_\e^+ & \text{for}\ x=\xi_1 \w_1 + \xi_2 \w_2 \colon \xi_1\ge j'{+}1 \,.
\end{cases}
$$
For $\eps$ sufficiently small, this function satisfies $\bar v_\e \in \A_{\rho,\infty}(\Om_{k,\infty})$ because $\det\D \bar v_\e>0$ a.e.\ in $\Om_{k,\infty}$.
Indeed, the images of the triangles of $\TT_\rho$ contained in $\Subdo_\e^-$ maintain the correct orientation:
if on the contrary this did not happen, one would get $v_\e(x) -(R_\e^- x + c_\e^-)\ge\frac{\sqrt3}{2}$ for some vertex $x$ of such a triangle.
The same holds for the triangles contained in $\Subdo_\e^+$.
This contradicts \eqref{eq-small-energy} and shows $\bar v_\e \in \A_{\rho,\infty}(\Om_{k,\infty})$. 
Moreover, the interaction energy of a bond lying in $\Subdo_\e^-$ between two points $x=(j{+}1)\w_1+\xi_2\w_2$ and $y=j\w_1+\xi_2\w_2$ 
is controlled by
$$
\left(\big| \bar v_\e(x)-\bar v_\e(y)\big|-1\right)^2 =
\left(\big|  v_\e(x) - (R_\e^- y + c_\e^-)\big| - \big|R_\e^- y-R_\e^- x\big|\right)^2 \le
\big|v_\e(x) - (R_\e^- x + c_\e^-)\big|^2 \,.
$$
An analogous estimate holds in $\Subdo_\e^+$.
Therefore, by \eqref{gammah}, \eqref{rescaled-bound}, and \eqref{eq-small-energy} it follows that 
\bes
\begin{split}
\ga^\lambda(\rho,k) & \leq \E_{\infty}(\bar v_\e) \leq \E_\e(u_\e)+
C\left(\integ{\Subdo_\e^-}{\big| \D v_\e - R_\e^- \big|^2}{x}
+ \integ{\Subdo_\e^+}{\big| \D v_\e -\tfrac\lambda\rho R_\e^+ \big|^2}{x}\right)\\
& \leq \E_\e(u_\e)+ C \e\,,
\end{split}
\ees
which yields \eqref{gammaliminf} on letting $\e \to 0^+$.
\end{proof}
\subsection{Upper bound}\label{section:five}
The first step for the proof of the upper bound is the following proposition.
\begin{proposition}\label{thm2}
Let $\ut \in W^{1,\infty}(\Om_k;\R^2)$,
$d_2\in L^{\infty}(\Om_k;\R^2)$ be such that 
$\partial_{\w_2} \ut = \partial_{\w_2} d_2 = 0$ and
\begin{equation} \label{cohull}
(\partial_1 \ut \, | \, d_2)\in 
\begin{cases}
\co(SO(2))     &  \text{a.e.\ in}\ \Om_k^-\,,\\
\co(\frac\lambda\rho SO(2))   &  \text{a.e.\ in}\ \Om_k^+\,. 
\end{cases}
\end{equation}
Then there exists a sequence $\{\ut_\e\}\subset \widetilde\A_{\rho,\eps}(\Om_{k})$ such that
\begin{equation} \label{1328061}
\D \ut_\e A_\e^{-1}\weakst (\partial_1 \ut \, | \, d_2) \quad \text{weakly* in}\ L^{\infty}(\Om_k;\Mdd)\,,
\end{equation}
and
\be\label{1328062}
\limsup_{\e\to 0^{+}} \I_{\e}(\tilde u_\e)\leq\ga^\lambda(\rho,k) \,.
\ee
\end{proposition}
\begin{proof}
We first assume that $(\partial_1 \ut \, | \, d_2)$ is piecewise constant with 
$(\partial_1 \ut \, | \, d_2) \in SO(2)$ a.e.\ in $\Om_k^-$ and 
$(\partial_1 \ut \, | \, d_2) \in \tfrac\lambda\rho SO(2)$ a.e.\ in $\Om_k^+$.
In this case there exist $m,n\in\Z$, $m<0$, $n\ge0$,
$-L=a_m<a_{m+1}<\dots<a_{-1}<a_0 =0 <a_1<\dots<a_n<a_{n+1}=L$,
and $R_{i}\in SO(2)$ for $i=m,\dots,-1,0,\dots,n$ such that
\begin{equation}
\label{piececonst}
(\partial_1 \ut \, | \, d_2) = \sum_{i=m}^{-1}\chi_{_{\Subdo_i }}R_{i}+
 \sum_{i=0}^{n}\chi_{_{\Subdo_i }}\tfrac\lambda\rho R_{i}\,,
\end{equation}
where $\chi_{_{\Subdo_i }}$ is the characteristic function of the set 
$$
\Subdo_i := \{\xi_1\w_1 + \xi_2\w_2 \colon a_i<\xi_1<a_{i+1} \,,\ 0<\xi_2<k  \}\,. 
$$
Let $\{\sigma_{\e}\}$ be a sequence of positive numbers such that 
$\e \ll\sigma_{\e}\ll 1$,  and set
$$
\widehat \Subdo_i :=  
\{\xi_1\w_1 + \xi_2\w_2 \colon a_i{-}\sigma_\e<\xi_1< a_i{+}\sigma_\e \,,\ 0<\xi_2<k  \} \,,
$$
see Figure~\ref{parall2}.
\begin{figure}
\centering
\psfrag{d}{\dots}
\psfrag{-L}{\hspace{-.9cm}$-L{=}a_m$}
\psfrag{-2}{\hspace{-.2cm}$a_{m{+}1}$}
\psfrag{-1}{\hspace{-.2cm}$a_{-1}$}
\psfrag{0}{$a_0{=}0$}
\psfrag{1}{$a_1$}
\psfrag{2}{$a_n$}
\psfrag{L}{\hspace{-.2cm}$a_{n+1}=L$}
\psfrag{p-L}{$\Subdo_m$}
\psfrag{p-1}{\hspace{.3cm}$\Subdo_{-1}$}
\psfrag{p0}{$\Subdo_0$}
\psfrag{pL}{\hspace{.3cm}$\Subdo_n$}
\psfrag{q-2}{\hspace{-.2cm}$\widehat\Subdo_{m+1}$}
\psfrag{q-1}{\hspace{-.2cm}$\widehat\Subdo_{-1}$}
\psfrag{q0}{\hspace{-.2cm}$\widehat\Subdo_{0}$}
\psfrag{q1}{\hspace{-.2cm}$\widehat\Subdo_{1}$}
\psfrag{qL}{\hspace{-.2cm}$\widehat\Subdo_{n}$}
\includegraphics[width=.8\textwidth]{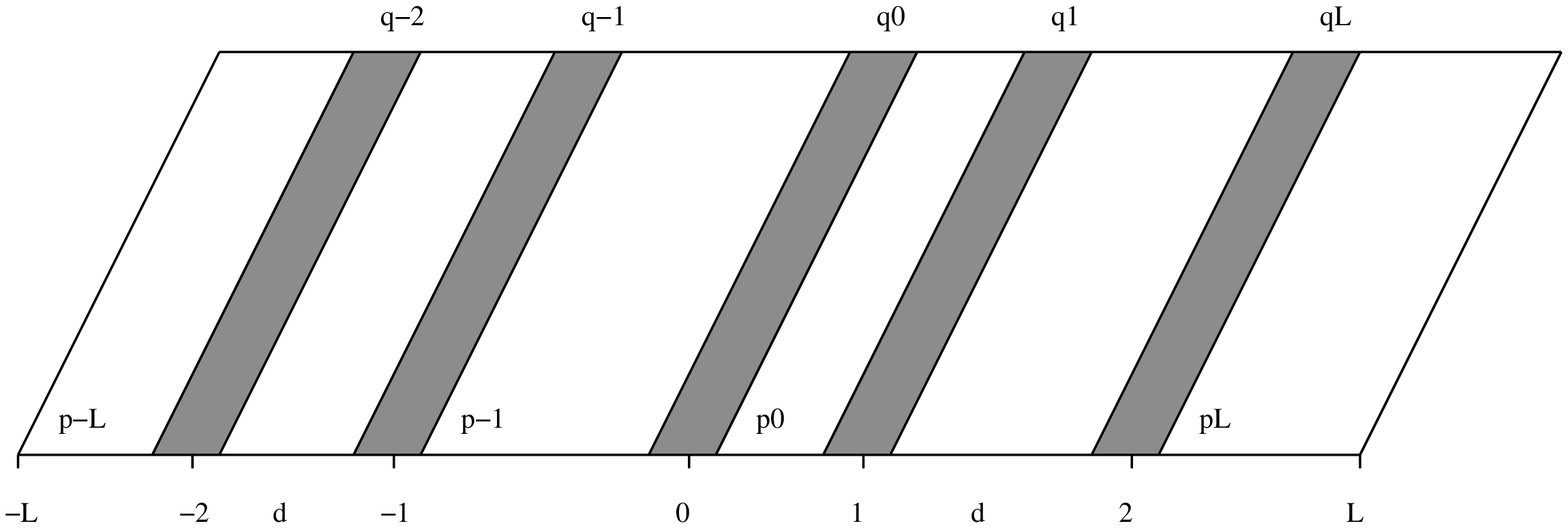}
\caption{The parallelograms $\Subdo_i$ and $\widehat\Subdo_i$.}
\label{parall2}
\end{figure}
We set
\begin{equation*}
\tilde u_{\e}(x):= 
\begin{cases}
R_{m}(A_\eps x)
&  \text{if}\ x\in \Subdo_m\setminus \widehat \Subdo_{m+1} \,,\\
R_{i}(A_\eps x) +c_\e^i
&  \text{if}\ x\in  \Subdo_i \setminus (\widehat \Subdo_i \cup \widehat \Subdo_{i+1})\,,
\quad i=m+1,\dots,-1\,,\\
\tfrac\lambda\rho R_{i}(A_\eps x)+c_\e^i
&  \text{if}\ x\in  \Subdo_i \setminus (\widehat \Subdo_i \cup \widehat \Subdo_{i+1})\,,
\quad i=0,\dots,n-1\,,\\
\tfrac\lambda\rho R_{n}(A_\eps x) +c_\e^n
& \text{if}\  x\in  \Subdo_{n}\setminus \widehat \Subdo_{n}\,,
\end{cases}
\end{equation*} 
where the constant vectors $c_\e^i$, $i=m+1,\dots,-1,0,\dots,n$ will be chosen below.
\par
In order to define $\tilde u_{\e}$ in the set 
$\widehat \Subdo_{0}$, we proceed in the 
following way. 
Let $\eta>0$. By definition of $\ga^\lambda(\rho,k)$,  see \eqref{gammah}, and by 
Proposition \ref{invariance},   
there exist $M>0$ and 
$v\in \A_{\rho,\infty}(\Om_{k,\infty})$ such that 
$$
 \D v=R_{-1}  \ \text{for} \ x_1\in(-\infty,-M)\,,\, \quad \D v=\tfrac\lambda\rho R_{0} \ \text{for} \ x_1\in(M,+\infty)
$$
and
\bes
\E_{\infty}(v)  \leq \ga^\lambda(\rho,k) + \eta \,.
\ees
Next define $\tilde u_{\e}$ in the set $\widehat \Subdo_{0}$ as 
$$
\tilde u_{\e}(x):= \eps v(\tfrac{1}{\eps}A_\eps x) + l_\e^0\,,
$$
where the constant $l_\e^0$ will be chosen below.
\par
In order to define  $\tilde u_{\e}$  in  the sets $\widehat \Subdo_i $ ($i<0$),
we proceed as in the proof of
Proposition~\ref{invariance}. 
For $\eps$ sufficiently small, we construct a smooth function 
$R_\e:\R\to SO(2)$ such that $R_\e(a_{i}{-}\sigma_{\e}{+}\frac{k\eps}{2})=R_{i-1}$ 
and 
$R_\e(a_{i}{+}\sigma_{\e})=R_{i}$
and define in the rescaled set $A_\eps\widehat \Subdo_i $
a piecewise affine function $w_\eps$ with a.e.\ positive Jacobian
such that
\begin{equation*}
w_\e (x) := 
\integlin{a_{i}-\sigma_{\e}+\frac{k\eps}{2}}{x_1}{R_\e(s) \begin{pmatrix}1 \cr 0 \end{pmatrix}}{s} + 
R_\e(x_1)\begin{pmatrix}0 \cr x_2 \end{pmatrix} 
+l_\e^i  \quad \text{for}\ x\in \L_{\rho,\eps}(k) \cap (A_\eps\widehat \Subdo_i ) \,,
\end{equation*}
where the constants $l_\e^i$ will be chosen below.
Notice that the rescaled parallelogram $A_\eps\widehat \Subdo_i$ has sides $\sigma_\e$ and $k\eps$.
We complete the definition of $\tilde u_\e$ by setting 
\begin{equation*}
\tilde u_\e (x) = w_\e (A_\e x) \quad \text{for}\ x\in \widehat \Subdo_i  \,.
\end{equation*}
The definition in  the sets $\widehat \Subdo_i $ ($i>0$) is analogous.
Next we choose the constants $c_\e^{i}$ and $l_\e^{i}$ so that the function $\tilde u_{\e}$ is continuous.
Arguing as in Proposition~\ref{invariance}, we see that $\tilde u_\eps\in \widetilde\A_{\rho,\eps}(\Om_{k})$ and that 
the cost of the transition in $\widehat\Subdo_i$ has order $O(\e/\sigma_\e)$,
so \eqref{1328061}--\eqref{1328062} hold.
\par
For the general case when \eqref{cohull} holds, we find a sequence of piecewise constant maps as
in \eqref{piececonst} which weakly* converges to $(\partial_1 \ut \, | \, d_2)$ in $L^\infty(\Om_k;\Mdd)$, and then conclude by approximation.
\end{proof}
The next theorem states that the domain of the $\Gamma$-limit 
of the sequence $\{  \I_\e \}$ is
\begin{equation}
\label{gammadomain}
\begin{split}
\A:=\big\{
u\in W^{1,\infty}(\Om_k;\R^2)\colon &
\partial_{\w_2} u  = 0 \ \text{a.e.\ in}\ \Om_k\,, \\
& |\partial_1 u |\leq 1 \ \text{a.e.\ in}\ \Om_k^- \,,
|\partial_1 u |\leq \tfrac\lambda\rho \ \text{a.e.\ in}\ \Om_k^+
\big\}
\end{split}
\end{equation}
and that the 
$\Gamma$-limit is constant in $\A$.
\begin{theorem}\label{thm3}
The sequence of functionals $\{  \I_{\e}\}$ 
$\Gamma$-converges, as $\e\to 0^{+}$, to the functional
\begin{equation}\label{eqthm3}
\I(u)=
\begin{cases}
\ga^\lambda(\rho,k)   &  \text{if}\ u\in\A\,,\\
+\infty   &  \text{otherwise,}
\end{cases}
\end{equation}
with respect to the weak* convergence in $W^{1,\infty}(\Om_k;\R^2)$.
\end{theorem}
\begin{proof}
The proof is divided into two parts.
\par
\emph{Liminf inequality.}
Let $ \ut_\e \in \widetilde\A_{\rho,\eps}(\Om_{k})$ be such that 
$\ut_\e\weakst  u$ weakly* in $W^{1,\infty}(\Om_k;\R^2)$. 
We have to prove that 
\bes
\I(u)\leq  \liminf_{\e \to 0^+} \I_\e(\ut_\e) \,.
\ees
We may assume that 
$\liminf_{\e \to 0^+} \I_\e(\ut_\e)  \le  C$. 
Then, by Theorem \ref{thm1}, there exist $\ut \in W^{1,\infty}(\Om_k;\R^2)$, 
$d_2\in L^{\infty}(\Om_k;\R^2)$  independent of $\w_2$, 
satisfying \eqref{charact}, 
such that $\D \ut_\e A_\e^{-1}$, up to subsequences, converges to $(\partial_1 \ut \, | \, d_2) $
 weakly* in $L^{\infty}(\Om_k;\Mdd)$ and \eqref{gammaliminf} holds.
Therefore $\partial_1u=\partial_1\ut$ a.e., and  the function $u$ does not depend on $\w_2$. 
Moreover, condition  \eqref{charact} implies that $u\in \A$.
\par
\emph{Limsup inequality.}
We have to show that for each $u\in W^{1,\infty}(\Om_k;\R^2)$
 there exists a sequence $\{\ut_\e\}\subset \widetilde\A_{\rho,\eps}(\Om_{k})$ such that 
$\ut_\e \weakst u$ weakly* in $W^{1,\infty}(\Om_k;\R^2)$ and 
\begin{equation}\label{ineq2}
\limsup_{\e\to 0^+} \I_\e (\ut_\e)\leq \I(u) \,.
\end{equation}
We can assume that $u\in\A$. 
It is possible to construct a measurable function  
$d_{2}\in L^{\infty}(\Om_k;\R^{2})$, independent of $\w_2$, such that 
\begin{equation*}
(\partial_1 u \, | \, d_2)\in 
\begin{cases}
\co(SO(2))     &  \text{a.e.\ in}\ \Om_k^-\,,\\
\co(\frac\lambda\rho SO(2))   &  \text{a.e.\ in}\ \Om_k^+\,. 
\end{cases}
\end{equation*}
For a detailed construction see e.g.\ \cite[Theorem 4.1]{mm}.
We now apply Proposition \ref{thm2} (with $\ut:=u$) to find a sequence 
$\{\ut_\e\}\subset  \widetilde\A_{\rho,\eps}(\Om_{k})$ such that
$\D \ut_\e A_\e^{-1}$ converges to 
$(\partial_1 u \, | \, d_{2})$ weakly* in $L^{\infty}(\Om_k;\Mdd)$, which 
implies $\nabla\ut_\eps\weakst\nabla u$ weakly* in $L^{\infty}(\Om_k;\Mdd)$
since $(\partial_1 \ut_{\e} \, | \, \partial_{\w_2} \ut_\e)$ converges to 
$(\partial_1 u \, | \,0)$. Moreover, \eqref{ineq2} holds. 
\end{proof}
\begin{remark}\label{rmk080213}
Recall the definition of the Delaunay triangulation in Section \ref{subsec:tri}.
In the degenerate cases we choose between two possible subdivisions of the quadrilateral cells,
see Figure \ref{fig-degen}.
We stress the fact that such a choice only influences the value of $\ga^\lambda(\rho,k)$, but not the qualitative analysis.
In particular, the weak limits of sequences $\ut_\e$ with equibounded energies in Theorem \ref{thm1} do not depend on such a choice, 
but only on the value of $\ut_\e$ at the points of the lattice, as one can compute.
\end{remark}
\section{Remarks and generalisations} \label{sec:rmk}\label{more-general}
In this section we mention some possible variations and extensions of the model,
whose proofs can be obtained from the previous ones with minor modifications.
\par
\bigskip
Recall the definition of the total interaction energy \eqref{eng-eps};
the energy relative to a bond across the interface is as in \eqref{eng-int}.
Different choices would be possible for which the previous results would still hold.
For example, one may consider
\bes
\begin{split}
\E_{\eps}^{\lambda,\tau}(u_\eps,\rho,k) & :=
\tfrac{1}{2}\!\!\! \sum_{\substack{\NN{x}{y}\\x\in \L_{1,\eps}^-(k) \\y\in\L_{1,\eps}^-(k)}}
\left(\Big|\frac{u_\eps(x)-u_\eps(y)}{\e}\Big|-1\right)^2 +
\tfrac{1}{2}\!\!\! \sum_{\substack{\NN{x}{y}\\x\in \L_{\rho,\eps}^+(k)\\y\in\L_{\rho,\eps}^+(k)}} 
\left(\Big|\frac{u_\eps(x)-u_\eps(y)}{\e}\Big|-\lambda\right)^2 \\
&+\sum_{\substack{\NN{x}{y}\\x\in \L_{1,\eps}^-(k)\\y\in\L_{\rho,\eps}^+(k)}} 
\left(\Big|\frac{u_\eps(x)-u_\eps(y)}{\e}\Big|-\mu\right)^2 
\end{split}
\ees
with $\mu>0$. 
\par
\bigskip
A possible modification concerns the distance between the two phases in the reference configuration.
Indeed, one can replace $\LL_1^-$ by
$\LL_\rho^- := \{ (-d(\rho)+\xi_1)\w_1+\xi_2\w_2\colon \xi_1,\xi_2\in\Z\,,\ \xi_1\le0\}$ with $d(\rho)>\frac13$ and $d(1)=1$.
(Before we considered for simplicity the case $d\equiv1$.)
The definition of $\LL_\rho^+$ is unchanged.
Here, the condition $d>\frac13$ guarantees that no points of $\LL_\rho^+$
lie in the interior of the circumcircle of any triangle with vertices in $\LL_\rho^-$.
This new choice changes the bonds between the neighbouring atoms across the interface
but the results of Sections \ref{sec:minimum} and \ref{sec:lim} still hold
(though the value of $\ga^\lambda(\rho,k)$, see \eqref{gammah}, will be different).
\par
\subsection{The case of next-to-nearest neighbours}
\begin{figure}[b]
\centering
\includegraphics[height=4cm]{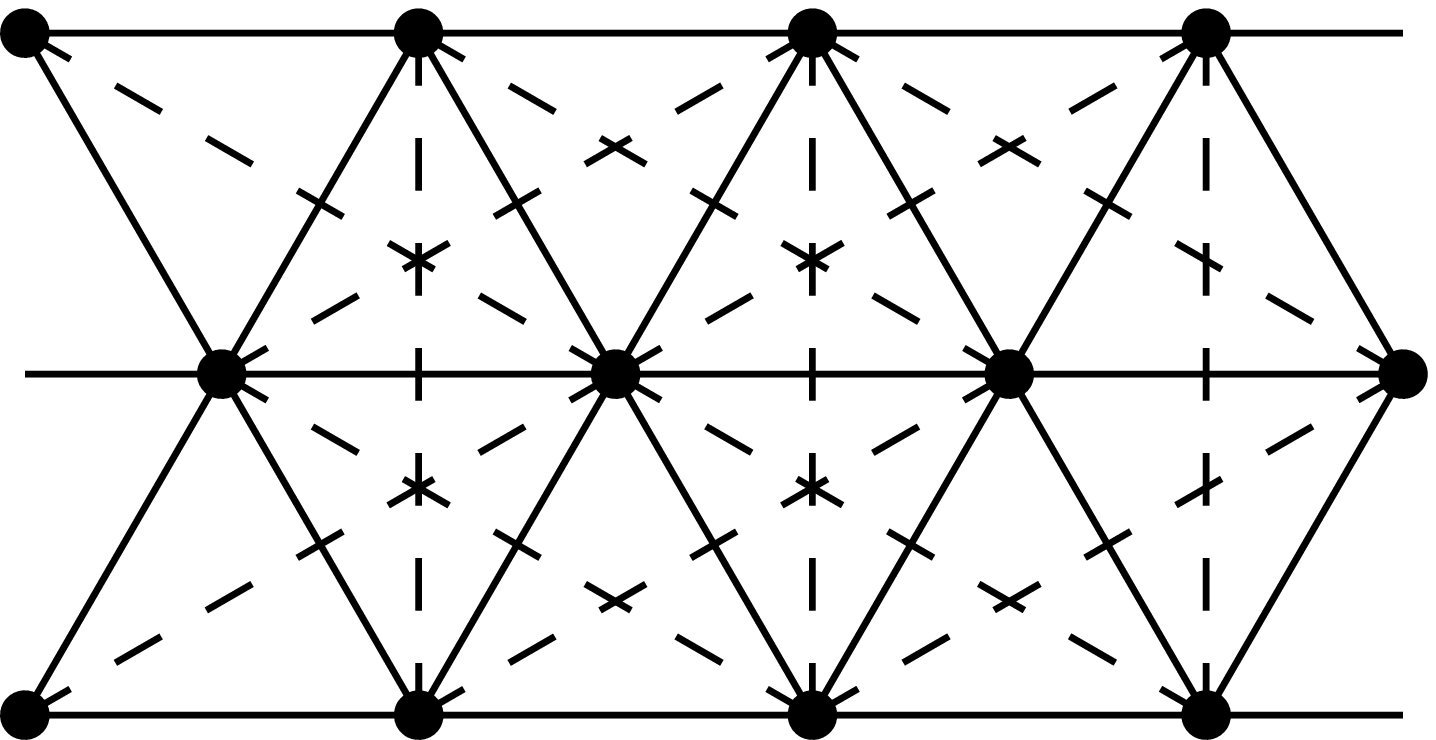}
\caption{Two-dimensional hexagonal Bravais lattice with nearest (solid) and next-to-nearest neighbours (dashed).
(Only the bonds between displayed atoms are presented.)}
\label{fig-nnn}
\end{figure}
With the same approach as above we can study a model with nearest and next-to-nearest neighbours.
In the two-dimensional hexagonal Bravais lattice generated by the vectors $\w_1$ and $\w_2$, 
two next-to-nearest-neigh\-bour\-ing atoms have distance $\sqrt3$, see Figure \ref{fig-nnn}.
\par
The definition of the next-to-nearest neighbours across the interface in $\LL_\rho$, see \eqref{1307251}--\eqref{1307253}, is
based on the notion of Delaunay pretriangulation (Definition \ref{def:Del-pre}).
In order to choose the next-to-nearest neighbours of a point $x\in\LL_\rho$, we drop its nearest neighbours from the lattice, setting
$$
\LL^*_\rho(x):=\LL_\rho\setminus\{y\colon \NN{x}{y}\} \,,
$$
and we consider the Voronoi diagram $\{C_x^*(y)\}_{y\in\LL_\rho^*(x)}$ associated with $\LL_\rho^*(x)$.
The next-to-nearest neighbours of $x$ are then the points $y\in\LL_\rho^*(x)$, $y\neq x$, such that $\H^1(C_x^*(x)\cap C_x^*(y))>0$;
in this case we write $\NNN{x}{y}$.
Away from the interface such definition induces the standard notion of next-to-nearest neighbours as recalled above.
\par
In $\LL_{\rho,\eps}(k)$, see \eqref{1307254}, two points $x,y$
are next-to-nearest neighbours if $x/\eps$, $y/\eps$ fulfill the corresponding property in the lattice $\LL_\rho$.
For $\eps>0$, $\lambda\in(0,1)$, $k\in\N$, $\rho\in(0,1]$, and $C_1,C_2>0$,
the energy of an admissible deformation $u_\eps\in\A_{\rho,\eps}$, see \eqref{adm}, is now
\bes
\begin{split}
\F_{\eps}^\lambda(u_\eps,\rho,k) &:=
\tfrac{C_1}{2}\!\!\! \sum_{\substack{\NN{x}{y}\\x\in \L_{1,\eps}^-(k) \\y\in\L_{\rho,\eps}(k)}}
\left(\Big|\frac{u_\eps(x){-}u_\eps(y)}{\e}\Big|-1\right)^2 +
\tfrac{C_1}{2}\!\!\! \sum_{\substack{\NN{x}{y}\\x\in \L_{\rho,\eps}^+(k)\\y\in\L_{\rho,\eps}(k)}} 
\left(\Big|\frac{u_\eps(x){-}u_\eps(y)}{\e}\Big|-\lambda\right)^2
\\
& + \tfrac{C_2}{2}\!\!\! \sum_{\substack{\NNN{x}{y}\\x\in \L_{1,\eps}^-(k) \\y\in\L_{\rho,\eps}(k)}}
\left(\Big|\frac{u_\eps(x){-}u_\eps(y)}{\e}\Big|-\sqrt3\right)^2 +
\tfrac{C_2}{2}\!\!\! \sum_{\substack{\NNN{x}{y}\\x\in \L_{\rho,\eps}^+(k)\\y\in\L_{\rho,\eps}(k)}} 
\left(\Big|\frac{u_\eps(x){-}u_\eps(y)}{\e}\Big|-\lambda\sqrt3\right)^2 \,.
\end{split}
\ees
%
All the results contained in the paper can be generalised to this case with minor modifications to the proofs.
\par
\begin{figure}[b]
\centering
\includegraphics[height=4cm]{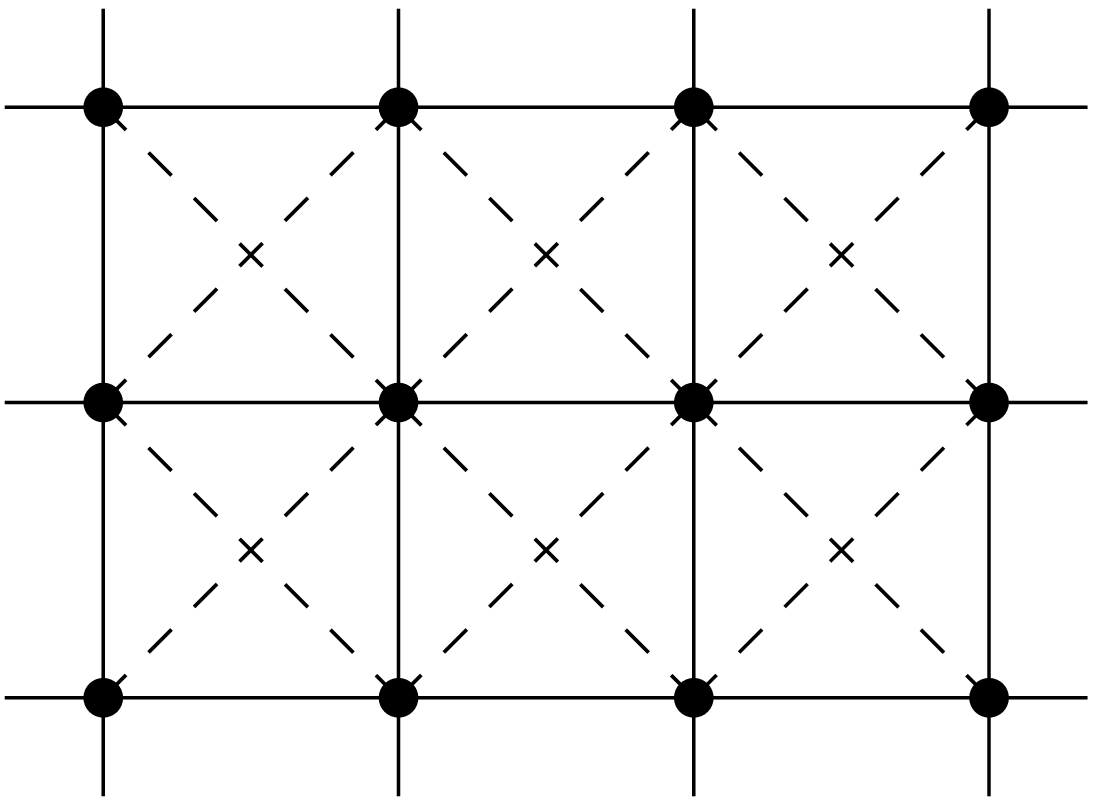}
\caption{
Square lattice with nearest (solid) and next-to-nearest neighbours (dashed).
}
\label{fig-square}
\end{figure}
Our approach can be applied also to other lattices enjoying the rigidity property,
e.g.\ the Bravais square lattice generated by the vectors $(1,0)$ and $(0,1)$
with nearest- and next-to-nearest-neighbour interactions.
In this case the nearest neighbours are just the edges of the Delaunay pretriangulation,
while the next-to-nearest neighbours are defined as for the hexagonal lattice.
This induces the standard notion away from the interface,
i.e., two nearest-neigh\-bour\-ing atoms have distance $1$
and two next-to-nearest-neigh\-bour\-ing atoms have distance $\sqrt2$,
see Figure \ref{fig-square}.
\subsection{The role of the non-interpenetration condition}
\label{rmk:non-int}
One of the points where we exploit the non-interpenetration condition, cf.\ \eqref{adm},
is the application of the rigidity estimate in the proof of Proposition \ref{lowerbound},
where we show that in the defect-free case $\ga^\lambda(1,k)$ grows as $k^2$.
If  instead the non-interpenetration is not assumed, then we can show that the sequence of 
minima of $\{  \I_{\e}\}$ grows linearly in $k$ also in the dislocation-free case $\rho=1$. 
Indeed, by introducing suitable reflections one 
can define deformations $u_\eps$ for which the energy of the bonds away from the interface is zero, while 
the total interaction energy grows linearly in $k$. 
This can be done by mapping all the atoms of $\L_{1,\eps}^-(k)$ to a triangle of side one
and all the atoms of $\L_{1,\eps}^+(k)$ to a triangle of side $\lambda$.
Alternatively, one can also enforce that $\nabla u_\eps=I$ in   
$\L_{1,\eps}^-(k)$ and $\nabla u_\eps=\lambda I$ in $\L_{1,\eps}^+(k)$ away from the 
interface like in Figure \ref{fig-counter}. 
\begin{figure}[b]
\centering
\includegraphics[width=.8\textwidth]{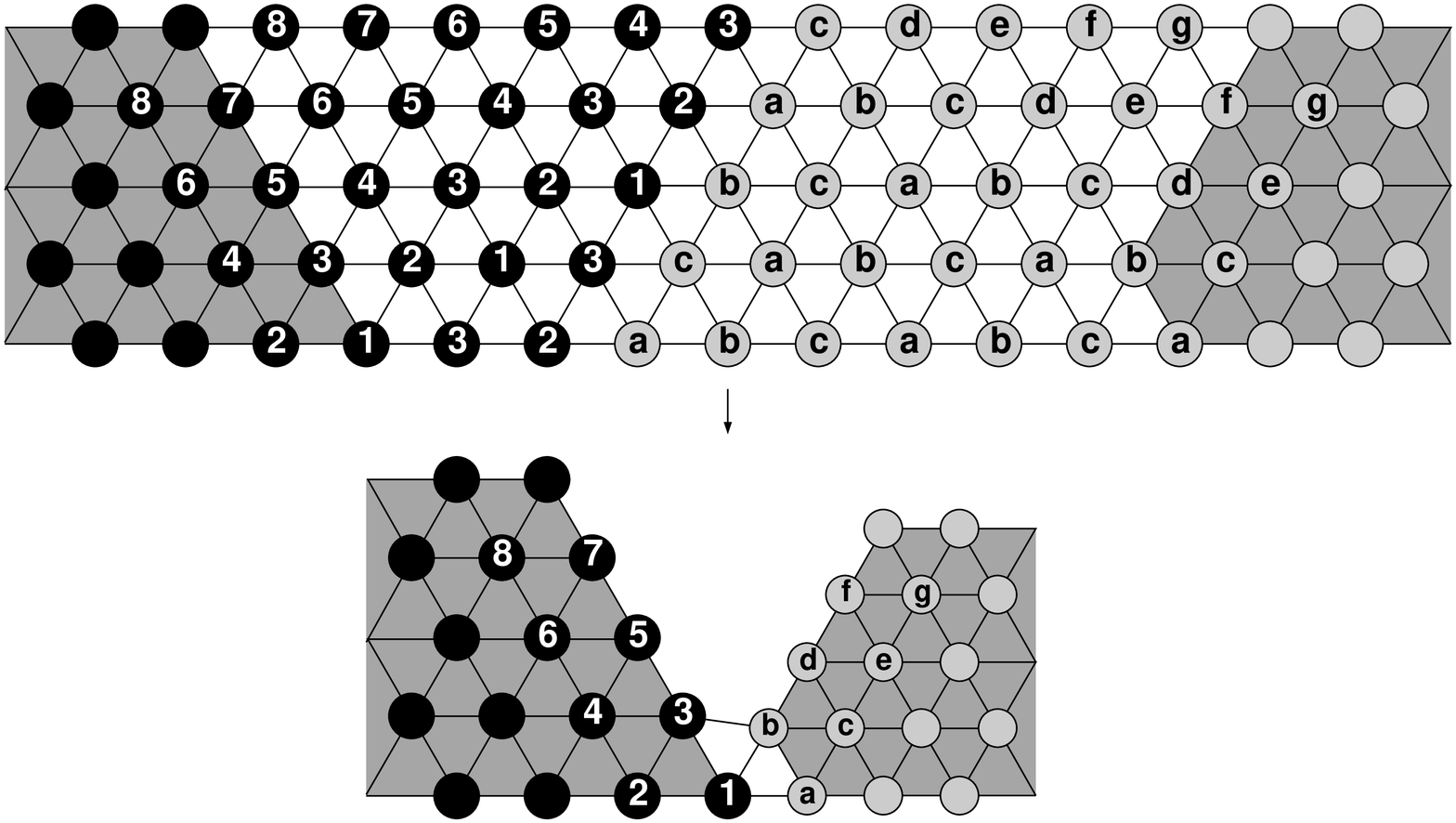}
\caption{A deformation that violates the non-interpenetration condition.
All points labelled with the same number or the same letter  are mapped 
to the same point. Moreover $\nabla u_\eps=I$ in the grey zone on the left and 
$\nabla u_\eps=\lambda I$ in the grey zone on the right part of the figure.
\\
When $k$ increases to $k+1$, the counterexample is extended in such a way that 
all the labels shown in the figure for $k$ are changed by moving the labels of the black points 
by one step to the left and the labels of the grey points by one step to the right. 
Accordingly, the grey zones move by one step to the left and right, respectively. 
At the interface new labels 1, 2, 3 and a, b, c are introduced consistently.
\\
For such deformations, the total interaction energy grows linearly in $k$.
}
\label{fig-counter}
\end{figure}
\section{A three-dimensional model}\label{3d}
All results presented so far essentially rely on the rigidity of the lattice
and can be generalised to three dimensions by choosing a lattice whose unit cells are rigid polyhedra.
For the sake of simplicity, we present here the extension of the model to the specific case of a face-centred cubic lattice,
which is related to the diamond structure of silicon nanowires \cite{Schmidtetal2010}.
We highlight the points where the treatment of the three-dimensional model is different from the two-dimensional case,
referring to the previous sections for all other details.
Other choices of rigid lattices are possible, including e.g.\ 
the (non-Bravais) hexagonal close-packed lattice, where the cells are also tetrahedra and octahedra.
The following discussion can be adapted to the latter case with minor modifications.
\par
The face-centred cubic lattice is the Bravais lattice generated by the vectors $\wu_1:=(1,0,0)$,
$\wu_2:=(\frac12,\frac12,0)$, and $\wu_3:=(0,\frac12,\frac12)$.
The nearest neighbours are those couples with distance $\frac{\sqrt2}{2}$;
such bonds are associated with the Delaunay pretriangulation (Definition \ref{def:Del-pre}),
which consists in a subdivision of the space into regular tetrahedra and octahedra, see Figure \ref{fig:fcc}. 
The diagonals of the octahedra correspond instead to next-to-nearest neighbours;
the distance between two  next-to-nearest neighbours is one.
\par
\begin{figure}[b]
\centering
\subfloat[]{
\includegraphics[width=.25\textwidth]{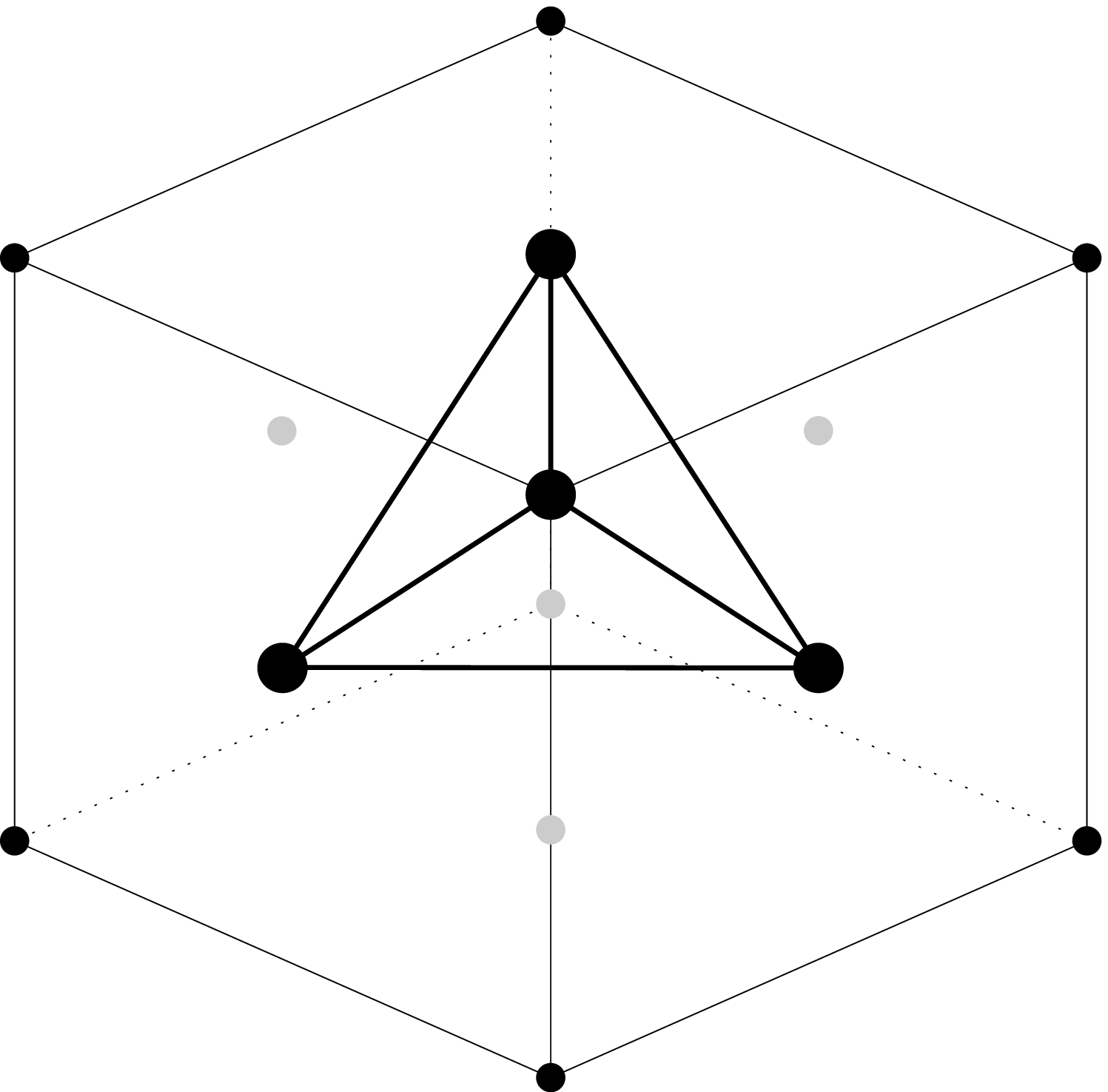}
\label{fig:subfig1}
}
\hfill
\subfloat[]{
\includegraphics[width=.25\textwidth]{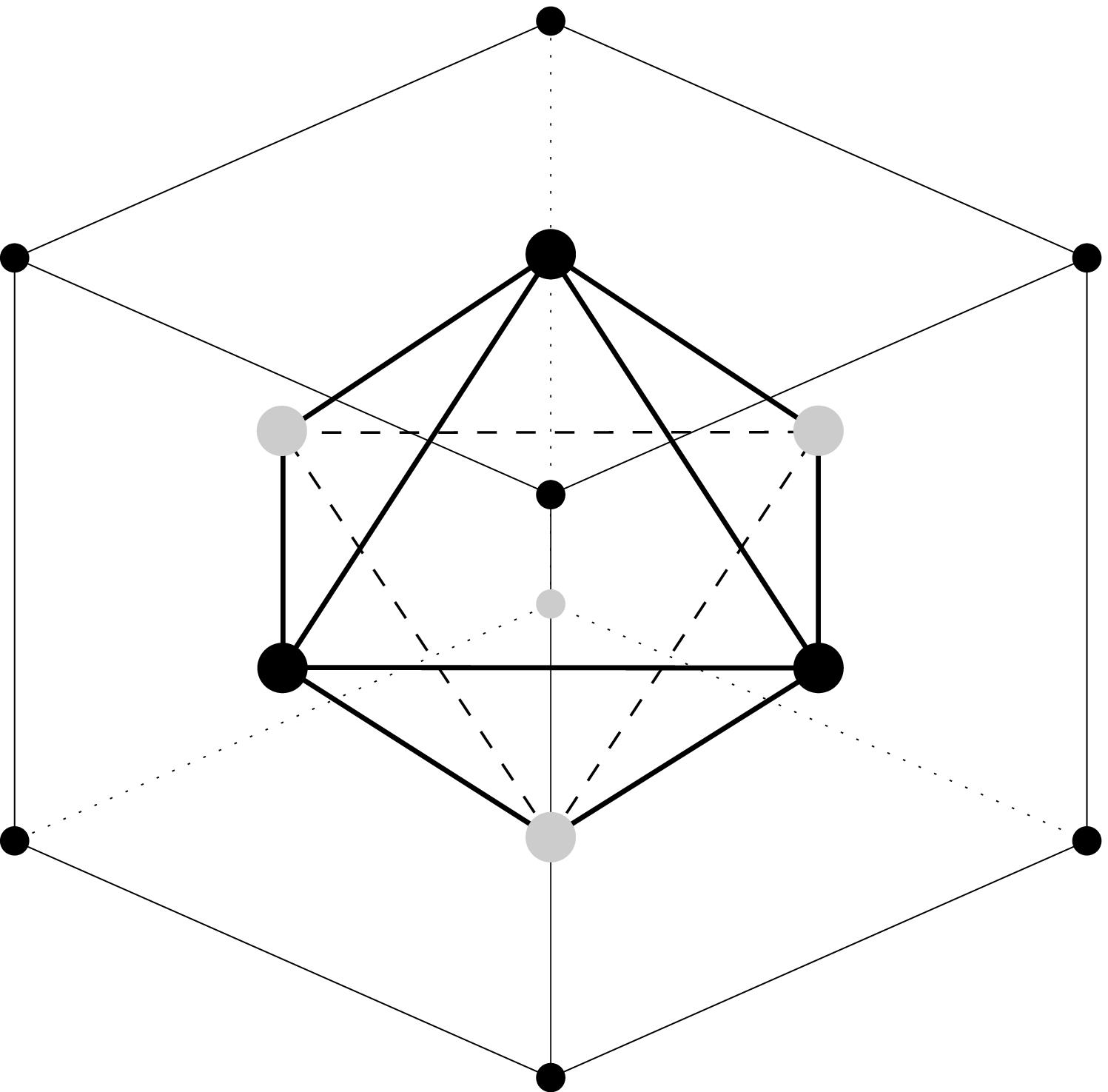}
\label{fig:subfig3}
}
\hfill
\subfloat[]{
\includegraphics[width=.25\textwidth]{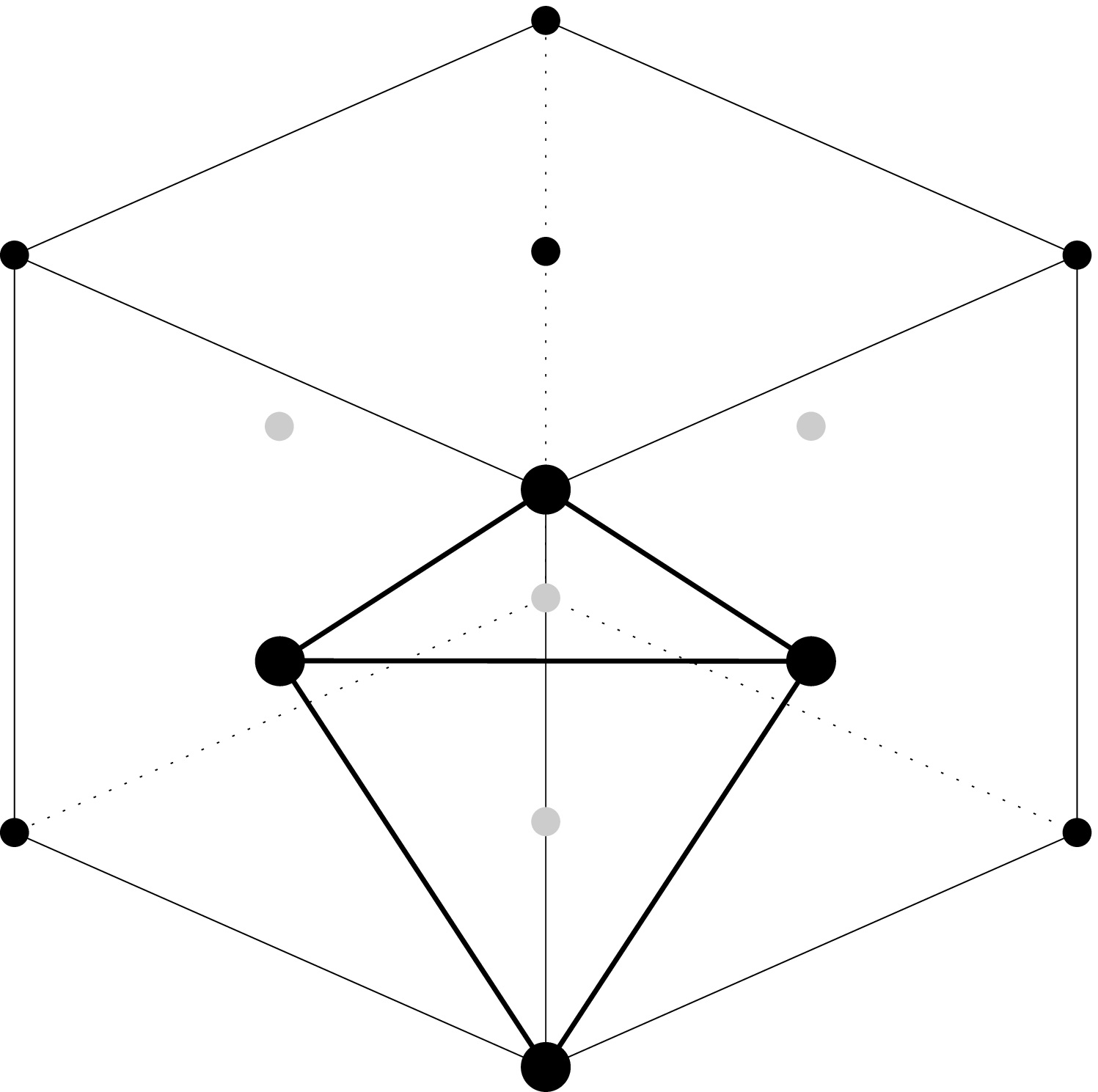}
\label{fig:subfig2}
}
\caption{Unit cell of the face-centred cubic lattice:
the nearest-neighbour structure of the atoms provides a subdivision of the space into tetrahedra and octahedra.
(\textsc{a}) tetrahedron;
(\textsc{b}) octahedron;
(\textsc{c}) quarter of octahedron.
Grey dots denote points lying on the hidden facets.
}\label{fig:fcc}
\end{figure}
Both the tetrahedron and the octahedron are rigid convex polyhedra.
By rigid convex polyhedron we mean that if the lenghts of the edges of the polyhedron are fixed,
then the polyhedron is determined up to rotations and translations,
under the assumption that the polyhedron itself is convex.
We recall that by the so-called Cauchy Rigidity Theorem, a convex polyhedron is rigid if and only if its facets are triangles.
\par
For fixed $\rho\in(0,1]$, the biphase atomistic lattice is the following:
\beas
\L_1^- &:=& \{ \xi_1\wu_1+\xi_2\wu_2+\xi_3\wu_3 \colon \xi_1,\xi_2,\xi_3\in\Z \,, \ \xi_1<0 \} \,, \\
\L_\rho^+ &:=& \{ \xi_1\wu_1+\xi_2\wu_2+\xi_3\wu_3 \colon \xi_1,\xi_2,\xi_3\in\rho\Z \,, \ \xi_1\ge0 \} \,, \\
\LL_\rho&:=&\LL_1^-\cup\LL_\rho^+\,.
\eeas
Notice that the interfacial planes are two-dimensional hexagonal Bravais lattices, see Figure \ref{fig:fcc-int}.
\begin{figure}
\centering
\includegraphics[width=.7\textwidth]{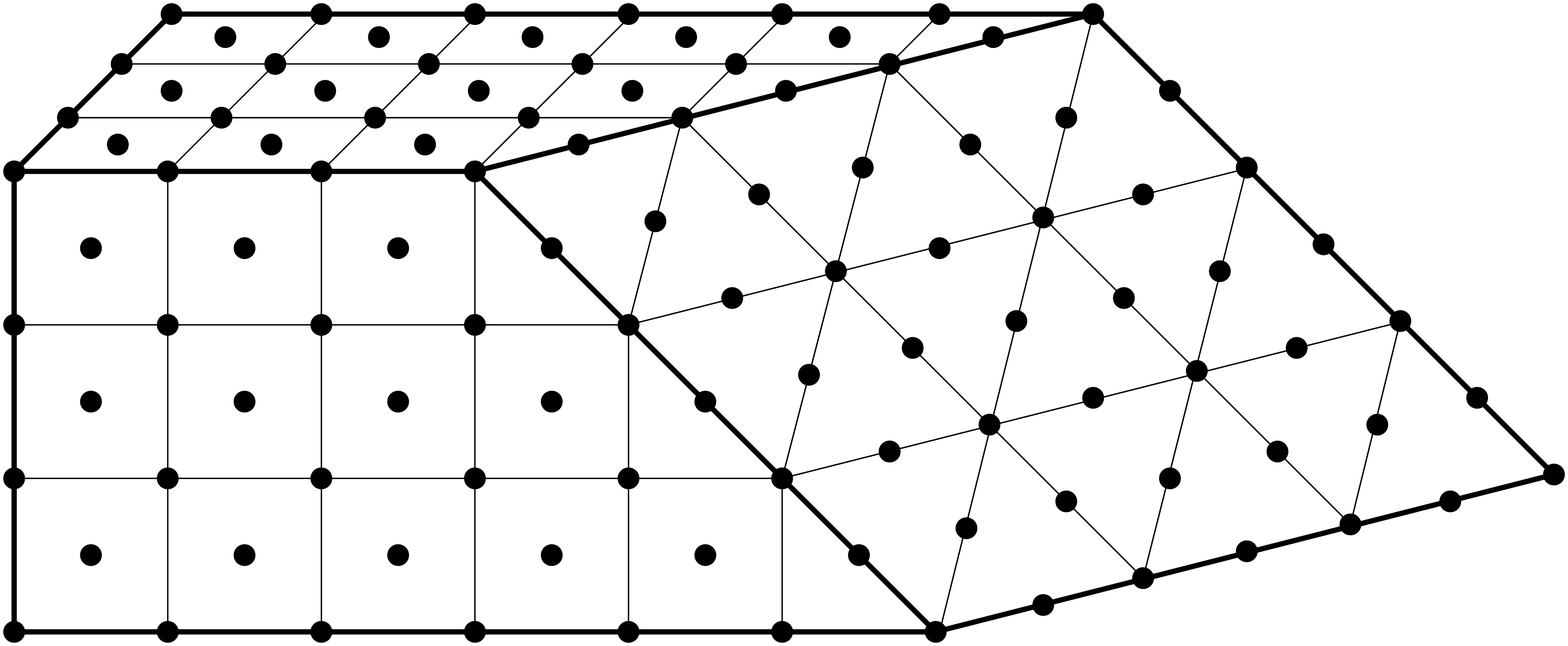}
\caption{By cutting the face-centred cubic lattice on a certain transverse plane, one obtains a two-dimensional hexagonal Bravais lattice.
Notice that the lines indicate the geometry of the setting and not the bonds between nearest neighbours.}
\label{fig:fcc-int}
\end{figure}
For $\rho\neq1$ the nearest neighbours are chosen as in the two-dimensional case.
Specifically, we consider the (unique) Delaunay pretriangulation $\TT'_\rho$ (Definition \ref{def:Del-pre}),
which is rigid away from the interface.
The partition $\TT'_\rho$ may contain polyhedra with some quadrilateral facets
across the interface between the two materials:
in this case we say that $\TT'_\rho$ is degenerate and refine it further, in order to obtain a tessellation in rigid polyhedra.
\par
More precisely, a polyhedron of $\TT'_\rho$ across the interface is one of the following:
\begin{enumerate}
\item a tetrahedron, irregular if $\rho\neq1$ (with three vertices in one of the two  lattices and one vertex in the other one, or two vertices in each of the lattices),
\item a pyramid with trapezoidal base (with three vertices in one of the two  lattices and two vertices in the other one),
\item a pyramidal frustum with triangular base (three vertices in each of the lattices),
\item an octahedron, irregular if $\rho\neq1$ (three vertices in each of the lattices).
\end{enumerate}
This can be easily seen by recalling that the interfacial atoms lie on two parallel planes
consisting of two-dimensional hexagonal Bravais lattices, with parallel primitive vectors.
If $\rho=1$ we only have cases 1 and 4.
\par
In cases 1 and 4 we leave the cell as it is, without introducing further bonds.
In cases 2 and 3, the cell is not rigid since it contains (at least) one quadrilateral facet;
therefore we subdivide each of the quadrilateral facets into two triangles, thus adding a couple of nearest neighbours;
we regard the resulting triangles as separate facets.
In case 2, the pyramid is then divided into two (irregular) tetrahedra; in case 3, we obtain one (irregular) octahedron.
Notice that there are different ways of subdividing the polyhedra;
we choose the same for all of them.
\par
In this way we define a partition in (possibly irregular) tetrahedra and octahedra and call it the rigid Delaunay tessellation associated to $\LL_\rho$, denoted by $\TT_\rho$.
The nearest neighbours are the extrema of each edge of any polyhedron of the subdivision.
We remark that the same procedure was followed in the two-dimensional case (Section \ref{sec:not}), where the Delaunay pretriangulation
may contain quadrilaterals across the interface.
\par
The rigid Delaunay tessellation $\TT_\rho$ determines the bonds that enter the definition of the energy.
Given $L>0$, $\eps\in(0,1]$, and $k\in\N$, we define 
$\L_{\rho,\eps}(k):=(\eps\LL_\rho)\cap\overline\Om_{k\eps}$, where
$$
\Om_{k\eps} := \{ \xi_1\wu_1+\xi_2\wu_2+\xi_3\wu_3\colon \xi_1\in(-L,L)\,,\ \xi_2,\xi_3\in(0, k\eps) \} \,.
$$
Like in the two-dimensional case, we set
\beas
\Om_{k\eps}^- &:=& \{ \xi_1\wu_1+\xi_2\wu_2+\xi_3\wu_3\colon \xi_1\in(-L,0)\,,\ \xi_2,\xi_3\in(0, k\eps) \} \,,\\
\Om_{k\eps}^+ &:=& \{ \xi_1\wu_1+\xi_2\wu_2+\xi_3\wu_3\colon \xi_1\in(0,L)\,,\ \xi_2,\xi_3\in(0, k\eps) \} \,,\\
\L_{1,\eps}^-(k) &:=& \{\xi_1\wu_1+\xi_2\wu_2+\xi_3\wu_3\in\L_{\rho,\eps}(k)\colon \xi_1<0\} \,,\\
\L_{\rho,\eps}^+(k) &:=& \{\xi_1\wu_1+\xi_2\wu_2+\xi_3\wu_3\in\L_{\rho,\eps}(k)\colon \xi_1\ge0\} \,.
\eeas
Two points $x,y$ in any of the previous lattices
are nearest neighbours if $x/\eps$, $y/\eps$ fulfill the corresponding property in the lattice $\LL_\rho$.
For every deformation $u_\eps\colon\LL_{\rho,\eps}(k)\to\R^3$ we define the total interaction energy
\bes
\E_{\eps}^\lambda(u_\eps,\rho,k) :=
\tfrac{1}{2}\!\!\! \sum_{\substack{\NN{x}{y}\\x\in \L_{1,\eps}^-(k) \\y\in\L_{\rho,\eps}(k)}}
\left(\Big|\frac{u_\eps(x)-u_\eps(y)}{\e}\Big|-1\right)^2 +
\tfrac{1}{2}\!\!\! \sum_{\substack{\NN{x}{y}\\x\in \L_{\rho,\eps}^+(k)\\y\in\L_{\rho,\eps}(k)}} 
\left(\Big|\frac{u_\eps(x)-u_\eps(y)}{\e}\Big|-\lambda\right)^2 \,,
\ees
where $\lambda\in(0,1)$.
\par
As before, in order to define the admissible deformations, we introduce piecewise affine functions.
Therefore, we need to refine $\TT_\rho$ to a proper Delaunay triangulation.
However, we do not change the definition of the nearest neighbours, i.e., we do not introduce new interactions in the energy.
Given a (possibly irregular) octahedron of $\TT_\rho$, we divide it into four irregular tetrahedra
by cutting it along one of the three diagonals.
We choose the diagonal starting from the vertex with the largest $x_1$-coordinate;
if two or three vertices have the same largest $x_1$-coordinate, we take among them the point with largest $x_2$-coordinate;
if two of such vertices have also the same largest $x_2$-coordinate, we take the one with the largest $x_3$-coordinate.
By repeating the process on every octahedron of $\TT_\rho$, we obtain a triangulation that we denote by $\TT_\rho^{(1)}$.
Other two triangulations $\TT_\rho^{(2)}$ and $\TT_\rho^{(3)}$ are obtained by repeating the same procedure,
but with different ordering of the indices, namely $x_2,x_3,x_1$ and $x_3,x_1,x_2$ respectively.
Given a function $u\colon \L_\rho\to\R^3$, we denote by $u^{(1)}$, $u^{(2)}$, and $u^{(3)}$ its piecewise affine interpolations
with respect to the triangulations $\TT_\rho^{(1)}$, $\TT_\rho^{(2)}$, and $\TT_\rho^{(3)}$, respectively.
\par
We define also $\TT_{\rho,\eps}:=\{\eps T\colon T\in\TT_\rho\}$ and $\TT_{\rho,\eps}^{(i)}:=\{\eps T\colon T\in\TT_\rho^{(i)}\}$
for $i=1,2,3$.
The set of admissible deformations is
\be\label{ad-3d}
\begin{split}
\A_{\rho,\eps}(\Om_{k\eps}):= \big\{ u_\eps\in C^0(\overline\Om_{k\eps};\R^3) \colon & u_\eps \ \text{piecewise affine,}\\ 
& \D u_\eps \ \text{constant on}\ \Om_{k\eps}\cap T\ \forall\, T\in\TT_{\rho,\eps}^{(1)}\,, \\
& \det \D u_\eps>0 \ \text{a.e.\ in}\ \Om_{k\eps}\,, \\
& u_\eps(\mathcal P) \ \text{is convex}\ \forall\, \mathcal P\in\TT_{\rho,\eps} \big\} \,.
\end{split}
\ee
As usual, the restriction of $u_\eps\in\A_{\rho,\eps}(\Om_{k\eps})$ to $\L_{\rho,\eps}(k)$ is still denoted by $u_\eps$.
Analogous definitions hold for $\Om_{k,\infty}$, $\L_{\rho,\infty}(k)$, and $\A_{\rho,\infty}(\Om_{k,\infty})$, see \eqref{1308061}--\eqref{1307232} and \eqref{130806}.
We will see that the limiting functional is independent of the choice of the triangulation $\TT_{\rho,\eps}^{(1)}$ in \eqref{ad-3d},
as suggested by the following remark.
\par
\begin{remark}\label{rmk:3d}
The assumption of convexity on the images of the octahedra of $\TT_{\rho,\eps}$ is needed to enforce rigidity:
without such an assumption an octahedron could be compressed without paying any energy. 
In fact, an elementary proof shows that $\det\D u_\eps^{(1)}>0$ and $u_\eps(\mathcal P)$ convex for every $\mathcal P\in\TT_{\rho,\eps}$
if and only if $\det\D u_\eps^{(i)}>0$ for every $i$.
Therefore, the notion of rigidity used in \eqref{ad-3d} is independent of the choice of the triangulation $\TT_{\rho,\eps}^{(1)}$
provided the image of each octahedron is assumed to be convex.
\end{remark}
The rigidity estimate of Lemma \ref{lemma-equiv} can be generalised to tetrahedra with a similar proof.
In the following, we set
$\w_1:=(1,0,0)$, $\w_2:=(\frac{1}{2},\frac{\sqrt{3}}{2},0)$, $\w_3:=\w_2-\w_1$,
$\w_4:=(\frac{1}{2},\frac{\sqrt{3}}{6},\frac{\sqrt{6}}{3})$,
$\w_5:=\w_4-\w_2$, and $\w_6:=\w_4-\w_1$.
\begin{lemma}\label{lemma:tetra}
There exists $C>0$ such that 
\bes
\dist^2(F,SO(3)) \leq  C \sum_{i=1}^6( |F\w_i|-1)^2  \quad \text{for every} \ F\in GL^+(3) \,. 
\ees
\end{lemma}
We prove the same estimate for each octahedron, subdividing it into four tetrahedra and using Remark \ref{rmk:3d} and Lemma \ref{lemma:tetra}.
In the following lemma we consider the octahedron $\mathcal O$ generated by the points
$P_1:=(0,0,0)$, $P_2:=(1,0,0)$, $P_3:=(0,1,0)$, $P_4:=(1,1,0)$, $P_5:=(\frac12,\frac12,\frac{\sqrt{2}}{2})$, 
and $P_6:=(\frac12,\frac12,-\frac{\sqrt{2}}{2})$.
%
\begin{lemma}
There exists $C>0$ such that
\bes
\dist^2(\D u,SO(3)) \leq  C \!\!\sum_{\NN{P_i}{P_j}}( |\D u(P_i-P_j)|-1)^2
\ees
for every $u\in C^0(\mathcal O;\R^3)$ such that $u$ is piecewise affine
with respect to the triangulation determined by cutting $\mathcal O$ along the vector $(1,1,0)$,
$\det\D u>0$,
and $u(\mathcal O)$ is convex.
\end{lemma}
\par
Thanks to the rigidity of the lattice, all the proofs presented in Section \ref{sec:lim} are extended to the present context.
The dimension reduction is performed with respect to the directions $\wu_2$, $\wu_3$.
As in the two-dimensional case, the limit of a sequence of discrete deformations with equibounded energy
does not depend on the triangulation chosen for the octahedra, see Remark \ref{rmk080213}.
The definition of the limiting functional and of its domain are the obvious extension of
\eqref{gammah}, \eqref{gammadomain}, and \eqref{eqthm3}.
We remark that the $\Gamma$-limit does not depend on the choice of the triangulation $\TT_{\rho,\eps}^{(1)}$ in \eqref{ad-3d},
since its formula only depends on the discrete values of the deformation, and not on its extension to the three-dimensional continuum.
\par
Arguing as in Section \ref{sec:gamma}, we prove that $\ga^\lambda(\rho,k,R)=\ga^\lambda(\rho,k,I)=:\ga^\lambda(\rho,k)$ for every $R\in SO(3)$
and the estimates
$$
C_1 k^3 \le \ga^\lambda(1,k) \le C_2 k^3
$$
and
$$
C_1' k^2\le\ga^\lambda(\lambda,k) \le C_2' k^2 \,.
$$
This proves that dislocations
are energetically preferred if the thickness of the nanowire is sufficiently large.
\section*{Acknowledgements} \noindent
The authors thank Roberto Alicandro for fruitful discussions
especially on the role of the non-interpenetration condition 
and for the construction of the example in Figure \ref{fig-counter}.
This research was supported by the DFG grant SCHL 1706/2-1. 
M.P.\ acknowledges the Department of Mathematics of the University of W\"urzburg,
to which she was affiliated when this work started.
\end{document}